\documentclass[12pt]{article}
\usepackage{epsfig,float}
\usepackage{amssymb,amsmath}
\usepackage{stmaryrd}
\def\theequation {\thesection.\arabic{equation}}
\makeatletter\@addtoreset {equation}{section}\makeatother

\def\ardlabel#1{\let\@currentlabel=\theequation\label{#1}}

\newtheorem{theorem}{Theorem}[section]
\newtheorem{lemma}[theorem]{Lemma}

\newtheorem{proposition}[theorem]{Proposition}
\newtheorem{property}[theorem]{Property}

\newtheorem{definition}[theorem]{Definition}
\newtheorem{remark}[theorem]{Remark}

\newenvironment{proof}{
    \noindent {\it Proof.}}{\hfill$\Box$
}

\newcommand{\be}{\begin{equation}}
\newcommand{\ee}{\end{equation}}

\begin{document}

\title{\bf Spectra of chains connected to complete graphs }

\author{J.-G. Caputo$^1$, 
G. Cruz-Pacheco$^2$, A. Knippel$^1$, 
P. Panayotaros $^2$ \\
{ \small $^1$
Laboratoire de Mathematiques,
INSA Rouen Normandie ,} \\
{\small Av. de l'Universit\'e,
76801 Saint-Etienne du Rouvray, France} \\
{\small $^2$
Depto. Matem\'{a}ticas y Mec\'{a}nica,
I.I.M.A.S.-U.N.A.M., } \\
{\small Apdo. Postal 20--726, 01000 M\'{e}xico D.F., M\'{e}xico } }

\date{\today}

\maketitle


\begin{abstract}

We characterize the spectrum of the Laplacian of graphs composed of one
or two finite or infinite chains connected to a complete graph. We show
the existence of localized eigenvectors of two types, eigenvectors
that vanish exactly outside the complete graph and eigenvectors that
decrease exponentially outside the complete graph. Our results also
imply gaps between the eigenvalues corresponding to localized and
extended eigenvectors.

\end{abstract}

\section{Introduction}

We study the spectrum of the Laplacian of finite and infinite graphs
obtained by joining complete graphs and chains. We focus on the  
existence and properties of eigenvectors that are localized 
in the nodes of the complete graph.

Our work is motivated by studies of linear and nonlinear 
mass-spring models of protein vibrations, where the interaction
between the amino-acids of the protein is encoded in a graph \cite{AT95,T96,JSPD07,MP16}. 
It has been suggested that oscillations localized in small subregions 
of certain types of proteins, particularly enzymes,
play a significant biochemical role, see references in \cite{PS11,CPMWF19}. 
Moreover, such oscillations occur in regions of higher density of the protein,  
where amino-acids interact with a larger number of neighboring 
amino-acids \cite{JSPD07,PS11}. The linear modes and frequencies of 
mass-spring models are connected to the eigenvectors and eigenvalues 
of the Laplacian of the graph describing the connectivity of the masses, 
see \cite{MP16}. 
Information on the frequencies and shape of the localized eigenvectors 
can be used to investigate weakly nonlinear modes \cite{MPO14,MP16,MP18}.

The above considerations justify the study of the spectrum of the
Laplacian of graphs that have the geometry of proteins. 
In \cite{MP16,MP18} we showed how localization in 
complete graphs connected to chain graphs suggests a mechanism of 
localization in proteins.
In the present work we consider some of the simplest examples of 
complete graphs connected to chains in order to obtain 
detailed information. We plan to use this information in a 
constructive approach to study more complicated graphs related to proteins.

We study first the Laplacian 
of a complete graph of $p \geq 6$ nodes joined to a chain of $q\geq 3$
nodes and show the existence of $p-2$ eigenvectors
with support in the complete graph component, these are referred 
to as ``clique eigenvectors''.
The clique eigenvectors are degenerate 
and have eigenvalue $p$, see Proposition \ref{internal-states-finite2}.
We also show the existence of 
an eigenvector that decays outside the complete graph
nodes and corresponds to an eigenvalue in the interval $(p,p+2)$, see 
Proposition \ref{edge-states-finite}.
This localized eigenvector, referred to as an ``edge eigenvector'', 
has its maximum amplitude at the node connecting the complete graph
to the chain and decays exponentially (up to a small error) in the chain.
Our analysis also gives the decay rate of the edge eigenvectors
and asymptotics for large $p$.
The remaining eigenvalues are in $[0,4]$. The corresponding eigenvectors, 
referred to as ``chain eigenvectors'', 
have small amplitude in the nodes of the complete graph.  
Similar statements apply to the case where $p \geq 6$, and $q = \infty$, 
see 
Propositions \ref{internal-states-infinite}, and \ref{edge-states-infinite}. 
The results on clique and edge eigenvectors and their eigenvalues 
are similar to the ones obtained for finite chains.     
The spectrum includes the interval $[0,4]$, corresponding to 
oscillatory nondecaying generalized eigenvectors.

We also analyze the spectrum of a complete graph of $p\geq 5 $ nodes 
joined to two chains of $q_1 $, $q_2\geq 3$ nodes 
respectively. We show the existence of $p-3$ clique eigenvectors and 
two edge eigenvectors, see 
Propositions
\ref{clique-evect-twoleg-inf},
\ref{2-leg-edge-states-finite}.
In the case $q_1 = q_2 = q$ we have 
one symmetric and one antisymmetric edge eigenvector.
The remaining eigenvalues are in $[0,4]$ and correspond to chain eigenvectors.
For $q_1 = q_2 = \infty$ we have similar results for the localized states, 
see Propositions \ref{infinite-clique-twoleg},
\ref{edge-states-infinite2}, with
one symmetric and one antisymmetric edge eigenvector.
The spectrum includes the interval $[0,4]$.

The constructive proofs of the edge eigenvectors for both finite and infinite chains
start with an exact computation of the clique eigenvectors. The
eigenvectors normal to the clique eigenvectors are examined by interpreting
the eigenvalue problem at the chain sites as a linear dynamical system,
with additional conditions at the boundary
of the chain. This analysis 
yields an algebraic equation for the edge eigenvalues. 
We obtain bounds for the edge eigenvalues by examining the roots
of the algebraic equations. Note that the proofs 
for the one- and two-chain finite and infinite problems follow the same pattern.
For two chains, the algebraic equations becomes more involved.
The analysis can be simplified by symmetry or
by using the Courant-Weyl estimates.
We note that the Courant-Weyl estimates, Lemmas \ref{weyl-one-leg}, \ref{weyl-two-leg},
give a good approximation of the spectrum for finite graphs. The dynamical 
approach is independent and leads to more precise results.

We also present additional numerical and asymptotic results, for instance
simplified expressions of the decay rate of edge eigenvectors for large $p$. 
We also note the possibility of embedded eigenvalues  
for small values of $p$. The results obtained for one or two chains
connected to a complete graph suggest conjectures
for the spectrum of 
a graph composed of many complete graphs connected by chains.

The article is organized as follows. Definitions and notation are
presented in section 2. Sections 3 and 4 contain respectively the 
analysis for a complete graph connected to one and two chains.
In Section 5 we present conjectures on the spectrum of
a graph composed of many complete graphs connected by chains.

\section{Definitions and Notation}

A finite undirected graph $G = ({\cal V},c)$ is defined by a set of vertices
${\cal V} \in \mathbb{Z} $ 
together with a connectivity function $c:{\cal V} \times {\cal V } \rightarrow \{0,-1\}$ satisfying 
$c_{ij} = -1$ if vertices $i,j$ are connected and $0$ otherwise.
From the connectivity function, one can build the Laplacian matrix of $G$ with
$|V|=n$ as the $n \times n$ matrix $L$ such that
$L_{ij}=c_{ij}=-1$ if $ij$ is an edge and
$L_{ii}= \sum_{j=1}^nL_{ij}$.
We assume the graph to be connected $L_{ii} > 0, \forall i$.
The degree $d_i$ is the number of connections of vertex $i$.

Since $L$ is symmetric and nonnegative, its eigenvalues $\lambda_k, k=1,\dots,n$ are 
real non negative 
and the eigenvectors $v^k$ can be chosen orthonormal. 
We can order the eigenvalues in the following way
$$\lambda_1 \ge \lambda_2 \ge \dots > \lambda_n=0  . $$
Note that only $\lambda_n$ is zero because the graph is connected \cite{crs01}.

We also allow the graph to be infinite, so that $G = ({\cal V},c)$
is a subset ${\cal V} $ of $\mathbb{Z}$ but assume  a finite degree
for each vertex.

We introduce some additional definitions for infinite graphs.
Consider the standard Hermitian inner product 
$ \langle f, g\rangle_c = \sum_{j \in {\cal I}} f_j g^*_j$, $f$, $g$ 
complex-valued functions on ${\cal V}$,
and the corresponding space $l^2_c = l^2_c({\cal V};{\mathbb{C}}) $ of 
$f$ satisfying 
$||f||^2_{l^2} = \langle f, f\rangle_c < \infty$.
We also consider the space $l^\infty_c$
of functions $f$ satisfying $ \sup_{j \in {\cal V}}  |f_j| < \infty$.
The real subspaces of real valued elements of $l^2_c$, $ l^\infty_c$
will be denoted by $l^2$, $l^\infty$.
The restriction of $ \langle \cdot, \cdot \rangle_c$ to elements of
$l^2 \times l^2$ defines an inner product in $l^2$, denoted by
 $ \langle \cdot, \cdot \rangle$.

Given a bounded linear operator $M$ in $l_c^2$, 
the residual set $\rho(M)$ of $M$ is the set of all 
$\lambda \in \mathbb{C}$ for which $M - \lambda I$
has a bounded inverse in $l^2_c$, 
$I$ the identity. The spectrum $\sigma(M)$ is the complement 
of  $\rho(M)$ in $\mathbb{C}$.  
The point spectrum $\sigma_p(M)$ of $M$ is the set 
of all $\lambda \in \mathbb{C}$ 
satisfying $M v = \lambda v$ for some $v$ in $l^2_c$.
Such $v \in l^2_c$ are also denoted as eigenvectors of $M$.
We have $\sigma_p(M) \subset\sigma(M)$.
The Laplacian $L$ of a graph of finite degree
is a bounded operator in $l^2_c$, and is also Hermitian, and nonnegative. 
We therefore have $\sigma(L) \subset [0,+\infty)$. 
  By the reality of $L$, $\sigma(L)$, we may seek eigenvalues of $L$
  in $l^2$.
  
In the case a finite graph,
$\sigma(L) = \sigma_p(L)$, i.e. the set of
eigevalues above. In the case of infinite graphs,
the spectrum of $L$ may be larger than the point spectrum.
We use the notion of the 
essential
spectrum $\sigma_e(M)$ of a bounded operator $M$ in $l^2_c$
defined as in \cite{KP13}, p.29, 
namely $\lambda \in \sigma(M)$
belongs to $\sigma_e(M)$ if $M-\lambda I$ either 
fails to be Fredholm, or is Fredholm with nonzero index,
see e.g. \cite{Kato76}, p.243, 
for a weaker condition, namely $M - \lambda I$
not semi-Fredholm.

In the case of the graphs studied below we see
examples of $u \in l^\infty_c$, $u \notin l^2_c$,
that satisfy $Lu = \lambda u$. Such $\lambda$, and $u$
may be referred to as generalized eigenvalues and eigenvectors of $L$
respectively. We can show that such $\lambda$ belong to
$\sigma_e(L)$, see e.g. \cite{KP13}, ch. 1.
Also, we see examples of
graphs for which $\sigma_e(L) \cap \sigma_p(L) \neq \emptyset$,
the corresponding eigenvalues may be referred to 
embedded eigenvalues.
We may distinguish cases where we have elements in $\sigma_p(L)$
is in the interior or the boundary of $\sigma_e(L)$.

We recall the standard definitions of a chain and a clique, together with
well-known properties.

\begin{definition}
  A chain $C_q$ is a connected graph of $q-1$ vertices
  whose Laplacian $L$ is a tri-diagonal matrix.
\end{definition}

The eigenvalues of the Laplacian of $C_q$ are given by
\begin{equation}
 \label{chain-eval}
\lambda_j =  4 \sin^2 \left [ \frac{\pi (q-1) -j}{2(q-1)} \right ], \quad j = 1, \ldots, q-1.
\end{equation}

\begin{definition}
A clique, $K_n$ is a complete graph with $n$ vertices. Its Laplacian
has $n$ on the diagonal $n$ and all other entries are equal to $-1$.
\end{definition}

The spectrum of a clique $K_n$ is well-known \cite{crs01}, we have
\begin{property}
A clique $K_n$ has eigenvalue $n$ with multiplicity $n-1$ and eigenvectors
$v^k = (1, 0,\dots , 0,-1, 0,\dots , 0,)^T,~~~k=1,\dots,n-1$ and 
eigenvalue $0$ for the constant eigenvector $v^n$.
\end{property}

\subsection{A chain connected to a clique }

We consider graphs formed by the association of clique $K_p$
and a chain $C_q$ denoted by 
$G = K_p \oplus C_q$, with $p \geq 3$, $q \geq 2$ positive integers, 
defined by the set 
\begin{equation}
\label{Kp-Cq-sites}
{\cal V} =  {\cal V}_{p,-} \cup {\cal V}_{q,+}, 
\quad 
{\cal V}_{p,-} =  \{-p + 1, - p + 2, \ldots, 0 \}, 
\quad 
{\cal V}_{q,+}  = \{1, 2, \ldots, q-1 \}, 
\end{equation}
and a connectivity matrix $c$ that satisfies 
\begin{equation}
\label{Kp-Cq-connectivity-1}
c_{ij} = -1, \quad \forall i,j \in {\cal V}_{p,-},  
\end{equation}
\begin{equation}
\label{Kp-Cq-connectivity-2}
c_{01} = c_{10} = -1,
\end{equation}
\begin{equation}
\label{Kp-Cq-connectivity-3}
c_{ij} = -1, \quad \forall i,j \in {\cal V}_{q,+} \quad\hbox{with}\quad |i-j| = 1,   
\end{equation}
and $c_{ij} = 0 $ for all other pairs $(i,j) \in {\cal V} \times {\cal V}$.
Equations \eqref{Kp-Cq-connectivity-1}-\eqref{Kp-Cq-connectivity-3} describe 
a complete graph of $p$ nodes, the set ${\cal V}_{p,-} $, joined to a chain 
of $q $ nodes, the set ${\cal V}_{q,+}$ with nearest-neighbor connectivity 
\eqref{Kp-Cq-connectivity-3}. 
The two sets are joined by  
\eqref{Kp-Cq-connectivity-2}.

\begin{property}
The Laplacian of a graph composed of a complete graph of $p$ nodes joined
joined to a chain of $q $ nodes is
$$L =   L'_{K_p} + L'_{C_{q}},         $$
where $L'_{C_{q}}$ is the augmented Laplacian of the chain $C_{q}$,
$L'_{C_{q}}= \begin{pmatrix} 0          & 0 \\ 0 &  L_{C_{q}}  \\ \end{pmatrix}$
and similarly $L'_{K_p}= \begin{pmatrix} L_{K_p}  & 0 \\ 0& 0       \\ \end{pmatrix}$.
\end{property}
We then write the graph as $ K_p \oplus C_q $.

\subsection{A motivational example }

We consider the graph $ K_6 \oplus C_4$ shown in Fig. \ref{ch3k6}.
\begin{figure}[H]
\centerline{
\epsfig{file=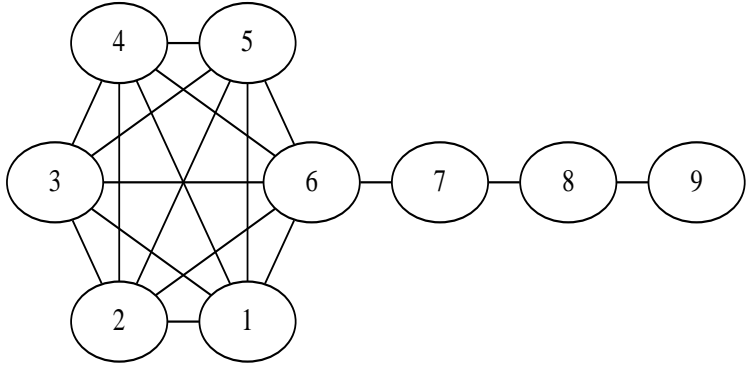,height=5cm,width=10cm,angle=0}
}
\caption{
A chain 3 coupled to the clique K$_6$}
\label{ch3k6}
\end{figure}

The Laplacian is 
\be\label{lapch3k6}
L =   L_{C_4} +L_{K_6} =
\begin{pmatrix}
5  &  -1 &   -1 &   -1 &  -1  &  -1 & . & . & . \\
-1 &  5  &   -1 &  -1  &  -1  &  -1 & . & . & . \\
-1 &  -1 &    5 &    -1&    -1&  -1 & . & . & . \\
-1 &   -1&  -1  & 5    &  -1  &  -1 & . & . & . \\
-1 &  -1 &  -1  &  -1  & 5    &  -1 & . & . & . \\
-1 &  -1 &  -1  &  -1  &  -1  &   6 & -1& . & .  \\
.  &  .  &   .  &  .   &   .  & -1  &  2& -1& . \\
.  & .   &  .   &  .   &  .   &   . & -1&  2& -1  \\
.  & .   &  .   &  .   &  .   &   . & . & -1& 1  \\
\end{pmatrix} ,
\ee
where $L_{C_4}$ is the Laplacian of an $n=9$ vertex graph
such that the last 4 vertices form a chain $q=4$,
$L_{K_6}$ is the Laplacian of an $n=9$ vertex graph
such that the first 6 vertices form a clique $p=6$ and where
the $0$'s are represented by $.$ for clarity.

The eigenvectors and corresponding eigenvalues are
%
%
%
{\small
\begin{table} [H]
\centering
\begin{tabular}{|l|c|c|c|c|c|c|c|c|r|}
   \hline
node & $v^1$  &    $v^2$ &   $v^3$   &  $v^4$  &  $v^5$ &   $v^6$  &  $v^7$    &   $v^8$   &  $v^9$ \\ \hline
1 &   0.1524  &      0   &     0     &     0   &     0  & -0.04879 &  0.098934 &  -0.23129 &  0.33333 \\
2 &   0.1524  & -0.707   &     0     & 0.707   &     0  & -0.04879 &  0.098934 &  -0.23129 &  0.33333 \\
3 &   0.1524  &      0   &   0.707   &-0.707   &     0  & -0.04879 &  0.098934 &  -0.23129 &  0.33333 \\
4 &   0.1524  &  0.707   &     0     &     0   & 0.707  & -0.04879 &  0.098934 &  -0.23129 &  0.33333 \\
5 &   0.1524  &      0   &  -0.707   &     0   &-0.707  & -0.04879 &  0.098934 &  -0.23129 &  0.33333 \\
6 &  -0.9198  &      0   &     0     &     0   &     0  &  0.10652 & -0.051079 &  -0.16999 &  0.33333 \\
7 &  0.19043  &      0   &     0     &     0   &     0  &  0.54399 &  -0.72369 &   0.18156 &  0.33333 \\
8 &-0.039105  &      0   &     0     &     0   &     0  & -0.75017 &  -0.29898 &   0.48499 &  0.33333 \\
9 & 0.0064791 &      0   &     0     &     0   &     0  &  0.34361 &   0.57908 &   0.65988 &  0.33333 \\
  &           &          &           &         &        &          &           &           &       \\ \hline
$\lambda$ & 7.0355 & 6   &     6     &     6   &     6  &   3.1832 &    1.5163 &   0.26503 &     0  \\ \hline
\end{tabular}
\caption{Eigenvectors and eigenvalues for the graph $G= C_4 \oplus K_6$.}
\label{tab1}
\end{table}
}

We see that $p-1=4$ of the $p=5$ eigenvalues of K$_6$ are preserved.
This is easy to see by padding with zeros 4 eigenvectors of 
K$_6$
$$ v^2=(0;0;0;-1;1;0;0;0;0)^T, ~~ v^3=( 0;1;-1;0;0;0;0;0;0)^T, $$
$$ v^4=(0;0;1;0;-1;0;0;0;0)^T ,  ~~ v^5=( 0;-1;0;1;0;0;0;0;0)^T$$
The five other eigenvectors $\{ v^1,v^6, v^7,v^8,v^9 \}$ have the property that
\be \label{vchain} v_k^i=x^i,  ~~k < 4  ,\ee
because these eigenvectors need to be orthogonal
to the eigenvectors of the clique $\{ v^2,v^3,v^4,v^5\}$. They are then 
constant in the clique section of the graph $k < 4$ except at the junction 
vertex.

For the graph studied above, these eigenvectors are plotted in Fig. \ref{rr}.
\begin{figure} [H]
\centerline{ \epsfig{file=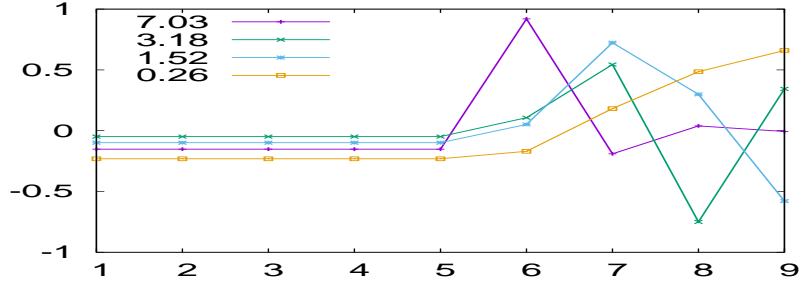,height=4 cm,width=12 cm,angle=0}}
\caption{Plot of the eigenvectors $\{ v^1,v^6, v^7,v^8 \}$ corresponding to 
eigenvalues $7.03,~3.18,~1.52,~0.26$.}
\label{rr}
\end{figure}

In the article, we prove the results illustrated in this
example, specifically
\begin{itemize}
\item the existence of $p-2$ "clique" eigenvalues $p$ for a graph
$G = K_p \oplus C_q$. The corresponding eigenvectors are non zero in the
clique region only.
\item the existence of one "edge" eigenvalue, strictly larger
than $p$, constant in the clique region and decaying
in the chain region. 
\item the existence of two "edge" eigenvalues, strictly larger
than $p$ for a clique $K_p$ connected to two chains $C_{q_1}, C_{q_2}$.
\end{itemize}

\section{A chain connected to a clique }

We study the graph $ K_p \oplus C_q $. 
First, we use the relation
for the Laplacian of $ K_p \oplus C_q $
$$L_{K_p \oplus C_q} = L'_{K_p}+L'_{C_q}$$
and the Courant-Weyl inequalities on the sum of two Hermitian matrices
to obtain bounds on the eigenvalues of $ K_p \oplus C_q $.

\subsection { Courant-Weyl inequalities} 

We have the following Courant-Weyl inequalities, see e.g. \cite{B96}.
\begin{proposition}
\label{courant_weyl}
Let $A$, $B$ symmetric $n \times n$ matrices, then
the Weyl inequalities are  
\begin{equation}
\label{gen-weyl}
\lambda_{k_1}(A) + \lambda_{k_2}(B) \leq \lambda_i(A+B) \leq  \lambda_{j_1}(A) + \lambda_{j_2}(B), 
\end{equation}
for all $i = 1, \ldots, n $, and 
all $k_1$, $k_2$, $j_1$, $j_2 \in \{1, \ldots, n\}$ 
satisfying $k_1 + k_2 = n + i$ and $j_1 + j_2 = i + 1$. 
\end{proposition}

For $A$, $B$ nonnegative this also implies 
\begin{equation}
\label{positive-weyl}
\lambda_i(A) \leq \lambda_i(A+B) \leq  \lambda_{j_1}(A) + \lambda_{j_2}(B) , 
\end{equation}
for all $i = 1, \ldots, n$, and 
all $j_1, j_2 \in \{1, \ldots, n\}$ 
satisfying $j_1 + j_2 = i + 1$ .

\begin{lemma}
\label{weyl-one-leg}
Let $L $ be the graph Laplacian of the graph $ K_p \oplus C_q $,
$p \geq 4$, $q \geq 4$.
Then $\lambda_1(L) \in [p,p+2]$, and $\lambda_j(L) = p$,
$\forall j \in \{2, \ldots, p-1\}$.  
Also $\lambda_p(L) \in (0,\hbox{min} \{ \lambda_1(L_{C_q}) + 2, p \})$
($\subset (0,p)$ if $p \geq 6$),   
and $\lambda_j(L) \in [0,4)$, $\forall j \in \{p+1,\ldots, p+ q - 1 \}$.  
\end{lemma}

\begin{proof}
  We will apply the Weyl inequalities \eqref{positive-weyl}
  using the decomposition $L = A + B$, 
where $A$ is the $(p+ q + 1) \times (p+ q + 1)$ block diagonal matrix with 
blocks $L_{K_p}$ and $L_{C_q}$. The only non-vanishing elements of 
$B$ are $B(p,p) = B(p+1, p+1) = 1$, $B(p+1,p) = B(p,p+1) = -1$. 

We then have $\lambda_j(A) = p$ if $j \in \{1, \ldots, p-1 \}$, and 
$\lambda_{p-1+k}(A) = \lambda_k(L_{C_q})  \in (0,4)$ for $k \in \{1, q-2\}$. Also  
$\lambda_{p + q - 2}(A) = \lambda_{p + q - 1}(A)  = 0$.
In addition, $\lambda_{1}(B) = 2$, and $\lambda_j(B) = 0 $,
$\forall j \in \{ 2, \ldots, p + q - 1 \}$.

We use $\lambda_j(A) \leq \lambda_j(A + B)$, $j = 1, \ldots, p + q - 1$, 
see \eqref{positive-weyl},
as lower bounds.

For the upper bounds we first note that    
$$ \lambda_1(A + B) \leq \lambda_1(A ) + \lambda_1(B) = p+2.$$
For $j \in {2, \ldots, p-1\ }$ we have  
$$ \lambda_j(A + B) \leq \hbox{min} \{\lambda_j(A) +
\lambda_1(B), \ldots, \lambda_1(A) \} = p, $$
since $\lambda_j(B) = 0 $, $\forall j \in \{ 2, \ldots, p + q - 1 \}$. 
Similarly 
$$ \lambda_p(A + B) \leq \hbox{min}
\{\lambda_p(A) + \lambda_1(B), \ldots, \lambda_1(A) \} 
= \hbox{min} \{ \lambda_1(L_{C_q}) + 2, p \}. $$
In the case $p \geq 6$ we also have $ \lambda_1(L_{C_q}) + 2 < 6 \leq p$, 
therefore $ \lambda_p(A + B) < p$.
 
For $j \in \{p+1,\ldots, p+ q - 1 \}$, i.e. 
$j = p-1+k$, $k \in \{2, \ldots, q-2\}$,  
we have
\begin{eqnarray}
\nonumber
\lambda_{p-1+k}(A + B)
& \leq  &
\hbox{min}\{\lambda_k(L_{C_q}) + \lambda_1(B), \lambda_{k-1}(L_{C_q}), 
\ldots, \lambda_{1}(A)\} \\
\nonumber
& = &
\hbox{min}\{\lambda_k(L_{C_q}) + 2, \lambda_{k-1}(L_{C_q}) \} < 4
\end{eqnarray}
using $ \lambda_{k-1}(L_{C_q}) < 4 < p$. 
Also 
$ \lambda_{p+q-2}(A + B) \leq \min\{2, \lambda_{{\tilde q}-1}(L_{C_q})\} < 4 $,   
and $ \lambda_{p+q-1}(A + B) = 0$.
The statement follows by combining the above with the lower bounds. 
\end{proof}

\subsection { Clique eigenvectors}

We show the existence of eigenvectors that 
are strictly zero outside the clique component $K_p$.

\begin{proposition}
\label{internal-states-finite2}
Let $L$ be the Laplacian of the graph $ K_p \oplus C_q $, $p \geq 3$, $q \geq 1$.
Then there exists a subspace $E_K$ of dimension $p-2$ such that
$v \in E_K$ satisfies $L v = p v$, and $v_j = 0$, $\forall j \notin \{-p + 1, \ldots, - 1 \}$.
\end{proposition}

\begin{proof}
Let $k \in \{2, \ldots, p-1 \}$ and define the vectors $v^k \in l^2$ by 
\begin{equation}
\label{intern-evec-one-leg}
	v^k_{-1} = 1, \quad v^k_{-k} = -1, \quad v^k_{j} = 0, \quad \forall j \in {\cal V}\setminus\{-1, -k\}. 
\end{equation}
We check that the $v^k$ are linearly independent.  
The statement follows by checking that 
$L v^k =  p v^k$, $\forall k \in \{2, \ldots, p-1 \}$.
Consider first $i \geq 1$. Then 
\begin{equation}
\label{case-in-out}
(L v^k)_{i} = \sum_{j \in \{-1, -k\}} L_{i,j} v^k_{j} = 0 =  p v^k_{i} 
\end{equation}
since $L_{ij} = c_{ij} = 0$ if $i \geq 1$, $j \leq - 1$ by 
\eqref{Kp-Cq-connectivity-1}-\eqref{Kp-Cq-connectivity-3}.
In the case $ i \in \{-(p-1), \ldots, 0\}\setminus\{-1,-k\}$ we have  
\begin{eqnarray}
\nonumber
(L v^k)_{i} &  = & \sum_{j \in \{-1, -k\}} L_{ij} v^k_{j} \\
\label{case-in-in}
& = & L_{i,-1} - L_{i,-k} 
	= c_{i,-1} - c_{i,-k} = 0 = p v^k_{i},
\end{eqnarray}
by \eqref{intern-evec-one-leg}, \eqref{Kp-Cq-connectivity-1}-\eqref{Kp-Cq-connectivity-3}.
Also, 
\begin{eqnarray}
\nonumber
(L v^k)_{-1} & = & \sum_{j \in \{-1, -k\}} L_{-1,j} v^k_{j} \\
\label{case-on-site-1}
& = &  L_{-1,-1} - L_{-1,-k} 
	=  (p-1) +1 = p =  p v^k_{-1}, 
\end{eqnarray}
and 
\begin{eqnarray}
\nonumber
(L v^k)_{-k} & = & \sum_{j \in \{-1, -k\}} L_{-k,j} v^k_{j} \\
\label{case-on-site-2}
& =  & L_{-k,-1} - L_{-k,-k} 
	= -1 + (-(p-1)) = -p =  p v^k_{-k}. 
\end{eqnarray}

\end{proof}

An alternative proof follows from noticing that the eigenvectors of $K_p$
that vanish at the sites connecting $K_p$ to the rest of the graph 
can be padded with zeros to form eigenvectors of $K_p \oplus C_q$,
see \cite{CKK19} and \cite{merris}.

The result extends to the case $q  = \infty$:
\begin{proposition}
\label{internal-states-infinite}
Let $L$ be the Laplacian of the graph $ K_p \oplus C_\infty $, $p \geq 3$. 
Then there exists a subspace $E_K \in l^2$ of dimension $p-2$ such that 
$v \in E_K$ satisfies $L v = p v$, and $v_j = 0$, $\forall j 
\notin \{-p + 1, \ldots, - 1 \}$.
\end{proposition}

The result follows by noticing that the
clique eigenvectors of finite chains
$ K_p \oplus C_q $
can be extended to be eigenvectors of 
$ K_p \oplus C_\infty $ by padding the sites beyond $q$ with zeros.

The next result is the existence of an ``edge eigenvector'', 
that is constant on the nodes of $K_p$ and decays exponentially in the
chain component. 

Let $E$ be a real subspace of $l^2$, then $E^{\perp}$ denotes 
its orthogonal complement with respect to $\langle \cdot, \cdot \rangle$.

\subsection{Edge eigenvector for a clique connected to an infinite chain}

We first introduce the transfer matrix formalism used
to simplify the problem in graphs that contain chains.

\subsubsection { Transfer matrix for a chain}

The equation $L v = \lambda v$ for the infinite chain ${\cal V} = \mathbb{Z}$ with
$c_{ij} = -1$ if $|i-j| = 1$, $c_{ij} = 0$ otherwise, is  
\begin{equation}
\label{eq-chain-form-inf}
-v_{j-1} + 2 v_j - v_{j+1} = \lambda v_j, \quad \forall j \in \mathbb{Z}.     
\end{equation}
%
%

Define $M_\lambda$, $\lambda \in {\mathbb{R}}$, and $z_j$ by 
\begin{equation}
\label{def-transf-matrix}
M_{\lambda} = \left[ 
\begin{array}{cc}
0 & 1 \\ 
-1 & -\lambda + 2 
\end{array}
\right],   ~~ z_j = \begin{pmatrix} v_j \\ v_{j+1} \end{pmatrix}.
\end{equation}

Then \eqref{eq-chain-form-inf}
is equivalent to
\begin{equation}
\label{transf-eq}
z_{j+1} = M_{\lambda} z_j, \quad \forall j \in \mathbb{Z}. 
\end{equation}
The eigenvalues $\sigma$ of $M_{\lambda}$ satisfy
$$\sigma^2 -(-\lambda+2)\sigma +1=0, $$
they are
\begin{equation}
\label{eval-transf-matrix}
\sigma_\pm = \frac{1}{2}[(-\lambda + 2) \pm \sqrt{(-\lambda + 2)^2 - 4}]. 
\end{equation}
The corresponding eigenvectors $v^\pm$ are
\begin{equation}
\label{evect-transf-matrix}
v^\pm = \begin{pmatrix} 1 \\ \sigma_{\pm} \end{pmatrix}.
\end{equation}
We have $\sigma_+ \sigma_- = 1$.
The discriminant of the equation in $\sigma$ is $\Delta = (2-\lambda)^2 -4$ and 
for $\lambda \in (0,4)$ we have $\Delta >0$ and elliptic dynamics for
\eqref{eq-chain-form-inf}.

We will be especially interested in the case 
$\lambda > 4$, where the dynamics is hyperbolic and $\sigma_+$ satisfies
the inequality $-1 < \sigma_+ < 0$.\\

\begin{remark}
  \label{large-lambda-sigma}
Also, that for $\lambda \gg 4$ we have
$$ \sigma_+ =
-{1 \over \lambda -2} + O\left( {1 \over (\lambda -2)^3 }\right),~~~~
\sigma_- = 2-\lambda + O\left( {1 \over (\lambda -2)^2 }\right) .$$
\end{remark}

\begin{proposition}
\label{edge-states-infinite}
Let $L$ be the Laplacian of the graph
$ K_p \oplus C_\infty $, $p \geq 5$, and let $E_K$ be as in Proposition
\ref{internal-states-infinite}.
Then $\lambda > 4 $
is an eigenvalue of $L$ with corresponding eigenvector $v \in l^2 \cap E_K^{\perp}$ 
if and only if 
$F(\lambda) = 0$, where
\begin{equation}
\label{one-leg-edge-eigenval-eq}
F(\lambda) =  (-\lambda + 1)\sigma_+  - (-\lambda + p)(-\lambda + 1) +  (p-1), 
\end{equation}
and $\sigma_+ = \sigma_+(\lambda)$ is as in \eqref{eval-transf-matrix}.
Furthermore, $F(\lambda) = 0 $ has exactly one solution in $(p,p+2)$,
and no solutions in $(4,p]\cup [p+2,+\infty)$.
\end{proposition}

\begin{proof}
   We construct $v \in  l^2 $ that satisfies
  $L v = \lambda v$, $\lambda > 4$.
  First, let $v$ be orthogonal to the span of the $p-2$ eigenvectors of
  Proposition \ref{internal-states-infinite} (i), then  
$$ - v_{-p+1} + v_{-1} = 0, \ldots, - v_{-2} + v_{-1} = 0,  $$
  therefore
\begin{equation}
   \label{def-c1-one-leg-inf}
v_{-p+1} = \ldots = v_{-1} = C_0   
\end{equation}
for some real $C_0$. Equation  
$L v = \lambda v$ at the nodes $k \in {\cal I}=\{ -p+1, \ldots, -1\}$ is 
\begin{equation*}
-\sum_{j \in {\cal I}\setminus\{k\}} v_{j} + d_{k} v_{k}  = \lambda v_{k},
  \end{equation*} 
or
\begin{equation*}
-(p-2) C_0 - v_{0} + (p-1) v_{k}  = \lambda v_{k}.
  \end{equation*} 
Using \eqref{def-c1-one-leg-inf}
we therefore have 
\begin{equation}
   \label{def-c2-one-leg-inf}
   v_{0} =  C_1 = (-\lambda + 1) C_0, 
\end{equation}
i.e. the same relation, for all $k \in \{ -p+1, \ldots, -1\}$. 

The condition 
$L v = \lambda v$ at the node $k = 0$ is 
  \begin{equation*}
-\sum_{j = - p+1}^{-1} c_{0,j} v_{j} + d_{0} v_{0} - v_{1} = \lambda v_{0},
\end{equation*}
  and reduces to
  \begin{equation}  
   \label{def-v1-one-leg-inf}
   v_{1} = (-\lambda + p)C_1 - (p-1)C_0 = [(-\lambda + p)(-\lambda + 1) - (p-1)]C_0.  
  \end{equation}
Furthermore, 
$L v = \lambda v$ at the nodes $j \geq 1$ is
  \begin{equation}  
\label{def-vgr1-one-leg-inf}
-v_{j-1} + 2 v_{j} - v_{j+1} = \lambda v_{j}, \quad \forall j \geq 1.   
 \end{equation}  
  We will show that \eqref{def-vgr1-one-leg-inf} implies a second condition on $v_{1}$,
  leading to 
an equation for $\lambda$. 
Letting
$$ z_j = \begin{pmatrix} v_j \\ v_{j+1} \end{pmatrix} ,~~j \geq 0, $$  
\eqref{def-vgr1-one-leg-inf} is equivalent to
\begin{equation}
\label{transf-eq-one-leg-inf}
z_{j+1} = M_{\lambda} z_j, \quad \forall j \geq 0,   
\end{equation}
with $M_\lambda$ as in \eqref{def-transf-matrix}.  
Therefore $z_0 = a v^+ + b v^- $, $a$, $b$ real, implies
that \eqref{def-vgr1-one-leg-inf} is equivalent to
\begin{equation}
\label{sol-transf-eq-one-leg-inf}
z_{n} = a \sigma_+^n v^+ + b \sigma_-^n v^-, \quad \forall n \geq 0,    
\end{equation}
with $\sigma_\pm$, $v^{\pm} $ as in \eqref{eval-transf-matrix},
\eqref{evect-transf-matrix} respectively.
The assumption $\lambda  >4 $ and \eqref{eval-transf-matrix} imply 
$|\sigma_-| > 1$. Therefore $v \in  l^2$ requires $b = 0$ in 
\eqref{sol-transf-eq-one-leg-inf}, in particular
we must require 
$$z_0 = \begin{pmatrix} v_0 \\ v_{1} \end{pmatrix} = a \begin{pmatrix} 1 \\ \sigma_+ \end{pmatrix}$$ for some real $a$, or equivalently 
\begin{equation}
\label{stable-dir-cond-one-leg-inf-1}  
  v_{1} = \sigma_+ v_{0} = \sigma_+ C_1. 
\end{equation}
By \eqref{def-c2-one-leg-inf},
we therefore have
\begin{equation}
\label{stable-dir-cond-one-leg-inf-2}  
  v_{1} = \sigma_+ (-\lambda + 1) C_0. 
\end{equation}
We may assume that $C_0 \neq 0$, otherwise $v$ vanishes at all nodes.
Then, 
comparing \eqref{def-v1-one-leg-inf}, \eqref{stable-dir-cond-one-leg-inf-2}
we have
\begin{equation}
\label{edge-eval-cond-one-leg-inf}
\sigma_+(-\lambda + 1) = (-\lambda + p)(-\lambda + 1) - (p-1), 
\end{equation}
which by $\sigma_+ = \sigma_+(\lambda)$ of \eqref{eval-transf-matrix}
is an equation for $\lambda$.

We first check that there is at least one solution $\lambda \in (p,p+2)$.
We let
\begin{equation}
\label{F-one-leg-inf}
F(\lambda) =  \sigma_+ (-\lambda + 1) - (-\lambda + p)(-\lambda + 1) +  (p-1). 
\end{equation}
$F$ is clearly continuous, moreover
$$ F(p) =  (p-1)(1 - \sigma_+) > 0, $$
since $ \sigma_+ \in (-1,0)$. Also
$$ F(p + 2) = - p (1 + \sigma_+) - (3 + \sigma_+) < 0,$$
by $ \sigma_+ \in (-1,0)$.

To see that there is only one root of $F$ in $(p,p+2)$ 
we check that $ F'(\lambda) < 0$ for all $\lambda \in (p,p+2)$.
Let $x = - \lambda + 2$, and examine
${\tilde F}(x) = F(\lambda(x))$ for $x \in (-p, -p+2)$.
Also let ${\tilde \sigma}_+(x) = \sigma_+(\lambda(x))$. 
By \eqref{F-one-leg-inf}
\begin{equation}
  \label{deriv-F-one-leg-inf}
  {\tilde F}'(x) = [{\tilde \sigma}(x) - x + 2 - p] + (x-1)({\tilde \sigma}'(x) - 1).
  \end{equation}
We claim that ${\tilde \sigma}'(x) < - 1/2$, $\forall x \in (-p, -p+2)$.
This follows from 
$$ {\tilde \sigma}'(x)  = \frac{1}{2} + h(x), \quad h(x) = \frac{x}{(x^2 - 4)^{1/2}}, $$
and $h(-p) < 0$, $h^2(-p)> 1 $, $h'(x) = - 4(x^2 - 4)^{-3/2} < 0$, $\forall x \in (-p, -p+2)$.
Then $ h(x) < -1$, $\forall x \in (-p, -p+2)$, and the claim follows.
Then \eqref{deriv-F-one-leg-inf}, 
$- x + 2 - p > 0 $, $x - 1 < - p +1$, and ${\tilde \sigma}(x) \in (-1,0)$
lead to  
$$ {\tilde F}'(x) > -1 + \frac{3}{2}(p-1) > 0, $$
$\forall x \in (-p, -p+2)$, by $p \geq 2$. It follows that  $ F'(\lambda) = -  {\tilde F}'(x) < 0$,
$ \forall \lambda \in (p,p+2)$.

To check that all solutions of \eqref{edge-eval-cond-one-leg-inf},
$\lambda > 4$, 
are in $(p,p+2)$,
consider first $\lambda \leq p $. Then 
$- (\lambda + 1) \leq p-1$ and by 
\eqref{edge-eval-cond-one-leg-inf} we have  
$$ \sigma_+ = p - \lambda + \frac{p-1}{\lambda - 1}
\geq 1, $$
contradicting $\sigma_+ \in (-1,0)$.
Assume now $\lambda \geq  p + 2 $, then
by \eqref{edge-eval-cond-one-leg-inf} we have 
$$ \sigma_+ \leq - 2  - \frac{p-1}{\lambda - 1} \leq - 1,$$
contradicting $\sigma_+ \in (-1,0)$.

\end{proof}

We have the following theorem for the essential spectrum 
of the graph Laplacian of $ K_p \oplus C_\infty $.
\begin{proposition} \label{especinf}
Let $L$ be the Laplacian of $ K_p \oplus C_\infty $ 
Then $ \sigma_e({L}) = [0,4]$.
\end{proposition}

\begin{proof}
Recall that the essential spectrum is invariant under finite rank 
perturbations, see e.g. \cite{Kato76}. The Laplacian of $ K_p \oplus C_\infty $, $p$ finite, 
is a finite rank perturbation of the Laplacian of the graph 
corresponding to $\mathbb{Z}^+$ with nearest neighbor connections.
The essential spectrum of this graph $[0,4]$. By the dynamics of the transfer matrix,
for $\lambda >4$ we have hyperbolic dynamics, 
therefore $ \lambda \rho(L)$. For $\lambda \in (0,4)$ we have oscillatory dynamics, and
generalized eigenvalues in $ v \in l^\infty$, $v \notin l^2$. Thus $(0,4) \in \sigma_e(L)$,
moreover $[0,4] \in \sigma_e(L)$ since $ \rho(L)$ is open.
\end{proof}

Finally, the whole spectrum of $ K_p \oplus C_\infty $ is given by
the following theorem.
\begin{proposition}
\label{spectrum-infinite}
Let $L$ be the Laplacian of the graph $ K_p \oplus C_\infty $, $p \geq 5$.
Then the spectrum of $L $ is a union of the disjoint sets $[0,4]$ (the essential
spectrum of $L$), and $\{\lambda, p \}$, with $ \lambda \in (p,p+2)$
(the point spectrum of $L$).
\end{proposition}

\begin{remark}
  We note that by
  Proposition \ref{internal-states-infinite}
  for $p = 3$ we have a clique eigenvector with eigenvalue
  $\lambda = 3 \in [0,4]$, i.e. an embedded eigenvalue. 
  For $p = 4$ we similarly have three clique eigenvectors at the boundary
  of the essential spectrum. In both cases there is 
  numerical evidence for
  an edge eigenvector outside $[0,4]$. 
\end{remark}

\subsection{Asymptotic estimates for large $p$}  

For large $p$, it is possible to use $F$ and the relations
established above to obtain asymptotics for $\lambda$ , $\sigma_+ $ and $C_0$.
From $F(\lambda)=0$, we get
\be\label{as1}
\sigma_+ - (-\lambda+p) +{ p-1 \over 1-\lambda}=0 . \ee
We can express $\sigma_+ $ as
$$\sigma_+ = {1 \over 2 } \left[  2-\lambda + |2-\lambda| \sqrt{1-{4 \over (2-\lambda)^2 }} \right ]$$
Let us assume $p \gg 1$. Since $\lambda \ge p$, we can expand $\sigma_+ $ in powers of
$\lambda$ and obtain
\be\label{as2}
\sigma_+ = -{1 \over \lambda -2} + O \left ( {1 \over (\lambda -2)^2} \right )
. \ee
Inserting  \eqref{as2} into \eqref{as1}, we get
$$\lambda = p + { 1-p \over 1-\lambda} + { 1 \over \lambda-2} $$
Solving step by step this expression, we obtain the final estimates
\begin{eqnarray}
\label{as3}	\lambda = p+1 + O\left ({1 \over p}\right ), \\  
\label{as4}	C_0= {1 \over 1-\lambda} = -{1 \over p} , \\
\label{as5}	\sigma_+ = -{1 \over \lambda -2} \approx {1 \over p-1} .
\end{eqnarray}
These expressions are reported in Table \ref{tab2} together with the
numerical solution for the graph $G= C_4 \oplus K_6$. As can be seen
the agreement is very good.
\begin{table} [H]
\centering
\begin{tabular}{|l|c|c|r|}
   \hline
                       & $\lambda$  & $\sigma_+ $ & $C_0$ \\ \hline
 Numerical solution    & 7.03    & -0.205    & -0.166 \\ \hline
 Theory                & 7.02    & -0.2    & -0.167  \\ \hline
\end{tabular}
\caption{\small\em Edge eigenvector: $\lambda$ , $\sigma_+ $ and $C_0$ for
the theory and the graph $G= C_4 \oplus K_6$. }
\label{tab2}
\end{table}

\subsection{Edge eigenvector for $ K_p \oplus C_q $ }



\begin{proposition}
\label{edge-states-finite}
Let $L$ be the Laplacian of the graph $ K_p \oplus C_q $, $p \geq 6$, $q \geq 3$,
and let $E_K$ be as in Proposition \ref{internal-states-finite2}.
Then $\lambda > 4 $
is an eigenvalue of $L$ with corresponding eigenvector $v \in E_K^{\perp}$ 
if and only if 
$F_q(\lambda) = 0$, where
\begin{equation}
\label{one-leg-edge-eigenval-eq-finite}
F_q(\lambda) =  (-\lambda + 1)\sigma_+  
\frac{1 + \sigma_+^{2 q - 3}}{1 + \sigma_+^{2 q - 1}} - (-\lambda + p)(-\lambda + 1) +  (p-1), 
\end{equation}
and $\sigma_+ = \sigma_+(\lambda)$ is as in \eqref{eval-transf-matrix}.
Furthermore, $F_q(\lambda) = 0 $ has exactly one solution in $(p,p+2]$,
and no solutions in 
$(4,p] \cup (p+2, +\infty)$. 
\end{proposition}

\begin{proof}
To construct $v $ that satisfies
  $L v = \lambda v$, $\lambda > 4$ and is 
orthogonal to the span of the $p-2$ eigenvectors of
Proposition \ref{internal-states-finite2} we  
argue as in the proof of Proposition \ref{edge-states-infinite}.
First, we must have 
\begin{equation}
   \label{def-c1-one-leg-finite}
v_{-p+1} = \ldots = v_{-1} = C_0   
\end{equation}
for some real $C_0$.  Arguing as in the proof of Proposition \ref{edge-states-infinite},   
$L v = \lambda v$ at the nodes $k = -p+1, \ldots, -1$, 
and $ v_{0} = C_1$ leads to the condition 
  \begin{equation}
   \label{def-c2-one-leg-finite}
   v_{0} =  C_1 = (-\lambda + 1) C_0, 
\end{equation}
and $L v = \lambda v$ at the node $k = 0$ 
reduces to
  \begin{equation}  
   \label{def-v1-one-leg-finite}
   v_{1} = [(-\lambda + p)(-\lambda + 1) - (p-1)]C_0.  
  \end{equation}

Furthermore, 
$L v = \lambda v$ at the nodes $ 1 \leq j \leq q- 2 $ is
\begin{equation}  
\label{def-vgr1-one-leg-finite}
v_{j-1} - 2 v_{j} + v_{j+1} = -\lambda v_{j},
\quad \forall j \in \{1, \ldots, q-2 \}.   
\end{equation}  
Letting $$ z_j =  \begin{pmatrix} v_j \\ v_{j+1} \end{pmatrix},~~  0 \leq j  \leq q - 2,$$  
\eqref{def-vgr1-one-leg-finite} is equivalent to
\begin{equation}
\label{transf-eq-one-leg-finite}
z_{j+1} = M_{\lambda} z_j, \quad \forall j \in \{1, \ldots, q-2 \},   
\end{equation}
with $M_\lambda$ as in \eqref{def-transf-matrix}.  
Therefore $z_0 = a v^+ + b v^- $, $a$, $b$ real, implies
\begin{equation}
\label{sol-transf-eq-one-leg-finite}
z_{n} = a \sigma_+^nv^+ + b \sigma_-^nv^-, \quad \forall n \in \{ 0, \ldots, q-1\}.   
\end{equation}
Evaluating at $n = q - 2$ using \eqref{evect-transf-matrix} we have  
\begin{equation}
\label{last-value-one-leg-finite-1}
\left[ 
\begin{array}{c}
v_{q-2} \\ 
v_{q-1} 
\end{array}
\right] = 
\left[
\begin{array}{c}
a \sigma_+^{q-2}  + b \sigma_-^{q-2}  \\ 
 a \sigma_+^{q - 1} + b \sigma_-^{q - 1} 
\end{array}
\right]. 
\end{equation} 
On the other hand, $L v = \lambda v$ at the node $q - 1$ is
\begin{equation}
\label{last-value-one-leg-finite-2}
v_{q - 1} = \frac{ v_{q-2}}{-\lambda + 1}. 
\end{equation} 
Compatibility of \eqref{last-value-one-leg-finite-1}, \eqref{last-value-one-leg-finite-2}
requires 
\begin{equation}
\label{last-value-one-compatibility}
\frac{1}{-\lambda + 1} [ a \sigma_+^{q-2}  + b \sigma_-^{q-2}] =  a \sigma_+^{q-1} + b \sigma_-^{q-1}.
\end{equation} 
We may assume that one of $a$, $b$ does not vanish, otherwise by 
\eqref{def-c1-one-leg-finite}, 
\eqref{def-c2-one-leg-finite},
\eqref{sol-transf-eq-one-leg-finite} we have the trivial vector.
Assuming $a \neq 0$, 
\eqref{last-value-one-compatibility} is equivalent to 
\begin{equation}
\label{ratio-ab-1}
\frac{b}{a} = \frac{\sigma_+^{q -2}}{\sigma_-^{q - 2}}
\left( \frac{(-\lambda + 1)\sigma_+ - 1}{1-(-\lambda + 1)\sigma_-} \right). 
\end{equation} 
$\sigma_\pm$ are eigenvalues of 
\eqref{def-transf-matrix} and therefore satisfy 
$\sigma^2 - (-\lambda +2) \sigma + 1 = 0 $. 
Using $-(-\lambda + 1) \sigma_{\pm} = - \sigma_{\pm}^2 + \sigma_{\pm} - 1 $
and $\sigma_+ \sigma_- = 1$ we 
simplify \eqref{ratio-ab-1} to 
\begin{equation}
\label{second-ratio-ab}
\frac{b}{a} = \frac{\sigma_+^{q - 1}}{\sigma_-^{q - 1}}
\left( \frac{\sigma_+ - 1 }{1 - \sigma_-} \right) = \sigma_+^{2 q - 1}. 
\end{equation} 
By $\lambda > 4$ we have $\sigma_+ \in (-1,0)$,
thus $ b \neq 0 $. Assuming $ b \neq 0$ we arrive 
at $a/b = \sigma_-^{2 q} \neq 0$, in a similar way.
Thus $a$, $b$ not both vanishing implies 
\eqref{second-ratio-ab} and $a$, $b \neq 0$. 

We now compare expressions  
\eqref{def-c2-one-leg-finite}, \eqref{def-v1-one-leg-finite}
for $v_0$, $v_1$, and $z_0 = a v^+ + b v^- $, using also \label{ab-ratio-oneleg-finite-2} 
for the ratio $b/a$, 
\begin{equation}
\label{joint-compatibility-one-leg-finite-1}
C_0 \left[ 
\begin{array}{c}
-\lambda + 1 \\ 
(-\lambda + p)(-\lambda +1) -(p - 1) 
\end{array}
\right] = 
a \left(
\left[
\begin{array}{c}
1   \\ 
\sigma_+ 
\end{array}
\right] + 
 \sigma_+^{2q-1} \left[
\begin{array}{c}
1  \\ 
\sigma_- 
\end{array} \right]
\right).
\end{equation} 
We may choose one of the components of $v$ freely. 
Choosing $C_0 = (-\lambda + 1)^{-1}$, 
the first component of \eqref{joint-compatibility-one-leg-finite-1} leads to 
$$ a = (1 + \sigma_+^{2 q})^{-1}. $$ 
Then the second 
component of \eqref{joint-compatibility-one-leg-finite-1} leads to 
\begin{equation}
\label{joint-compatibility-one-leg-finite-2}
(-\lambda + p) - \frac{p-1}{-\lambda + 1} =  \sigma_+
\frac{1 + \sigma_+^{2 q - 3}}{ 1 + \sigma_+^{2 q - 1}}. 
\end{equation} 
By \eqref{eval-transf-matrix} this is an equation for $\lambda$. It is precisely the 
equation   
$F_q (\lambda) = 0$, with $F_q$ as in 
\eqref{one-leg-edge-eigenval-eq-finite}.

We now examine the roots of $F_q(\lambda)$, $\lambda > 4$. 
Assume
first $4 < \lambda \leq p $ and $F_q(\lambda) = 0$. Then
the left hand side of \eqref{joint-compatibility-one-leg-finite-2} satisfies
$$ -\lambda + p -  \frac{p-1}{-\lambda + 1} >  \frac{p-1}{\lambda -1}  >    1 $$
by the hypothesis $\lambda \leq p $. One the other hand, $ \sigma_+ \in (-1,0)$
by $\lambda > 4$, 
therefore 
$$ \frac{1 + \sigma_+^{2 q - 3}}{ 1 + \sigma_+^{2 q - 1}} 
\in (0,1). $$
The right hand side of \eqref{joint-compatibility-one-leg-finite-2} is then 
in $(-1,0)$.  Thus $F_q$ has no roots in $(4,p]$.

By Proposition \ref{weyl-one-leg} we have $\lambda_1(L) \leq p+2 $, 
thus all solutions of $F(\lambda)$ must belong to the interval $(p,p+2]$.
It follows that $F(\lambda) = 0$ has exactly one solution in $(p,p+2]$,
otherwise we would have $\lambda_1(L) < p$, contradicting Proposition \ref{weyl-one-leg}.
\end{proof}

\subsection{Chain eigenvectors for $q$ finite}

In the arguments above, the condition $\lambda > 4$ was only used to
locate the roots of $F_q$. 
We can therefore use
The function $F_q(\lambda)$ to study the eigenvalues 
$0 \le \lambda \le 4 $. In this region, the roots $\sigma$ are
imaginary and on the unit circle. It is easier to describe them
using the phase 
\be\label{phase} 
\phi = {\rm atan} ({\sqrt{4-(2-\lambda)^2} \over 2-\lambda }) .\ee
From this expression, we can write $F_q$ as
\be\label{fphase}
F_q(\lambda)= (-\lambda +1){ \cos(\phi) + \cos(2(q-1) \phi) \over 1+\cos((2q-1) \phi)}
- (-\lambda +p)(-\lambda +1) +p-1 .\ee
This function of $\lambda$ is plotted in Fig. \ref{rbfc} for
the graph $C_4 \oplus K_6$.

\begin{figure} [H]
\centerline{ \epsfig{file=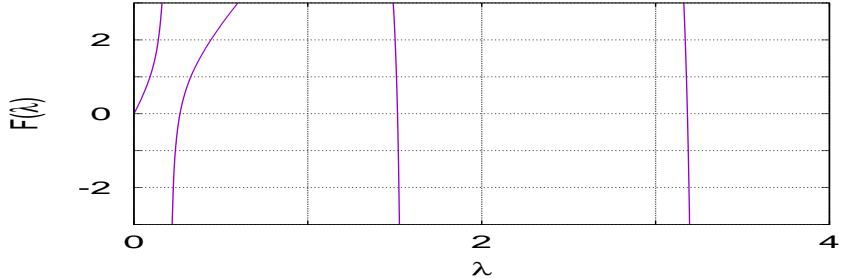,height=4 cm,width=12 cm,angle=0}}
	\caption{Plot of the function $F_q(\lambda)$ for $0 \le \lambda \le 4 $
	showing the three chain eigenvalues for the graph
	$C_4 \oplus K_6$. } 
\label{rbfc}
\end{figure}

The eigenvalues are the zeros of the function $F_q(\lambda)$ whose
graph is presented in Fig. \ref{rbfc}. Once these zeros are estimated,
relation \eqref{def-c2-one-leg-inf} allows to compute the ratio
of the eigenvector components at the edge and inside the clique.
The results are summarized in Table \ref{tab3} for the graph
        $C_4 \oplus K_6$ presented in section 2.2.
\begin{table} [H]
\centering
\begin{tabular}{|l|c|c|c|r|}
   \hline
$n$     & zero of $F_q(\lambda)$ & $\lambda_n$ & $C_0={1\over 1-\lambda_n}$  & $v_5^n/v_4^n $ \\ \hline
8       & 0.265033 & 0.26503  & 1.360                       & 1.361   \\
7       & 1.51622  & 1.5163   & -1.9365                     & -1.9368 \\
6       & 3.1832   & 3.1832   & -0.4580                     & -0.4580 \\
\hline
\end{tabular}
\caption{
Comparison of the zeros of the function $F_q(\lambda)$ and the chain eigenvalues $\lambda_n$
for the graph $C_4 \oplus K_6$ presented in section 2.2. The ratio of the eigenvectors
components at the edge and inside the clique are given in the last two columns. }
\label{tab3}
\end{table}
As can be seen, the agreement between the eigenvalues and the zeros of
$F_q(\lambda)$
is excellent. The ratios of the eigenvectors at the edge and inside the clique
also agree well with the ones of section 2.2. Without normalization, the eigenvector 
components $v_k^n$ in the chain section would be 
$$v_k^n = A_n \cos \left [ \lambda_n (k-5) \right ] + B_n \sin \left [ \lambda_n (k-5) \right ],$$
with $A_n,B_n$ chosen as to satisfy the boundary condition at the end of
the chain.

\subsection{Spectrum of $ K_p \oplus C_q $}

We have the following theorem summarizing
the spectrum of the Laplacian of $ K_p \oplus C_q $.
\begin{proposition}
\label{spectrum-finite}
Let $L$ be the Laplacian of the graph $ K_p \oplus  C_q $, $p \geq 3$.
Then the spectrum of $L $ consists of the eigenvalue $p$, of multiplicity $p - 2$,
a simple eigenvalue $ \lambda \in (p,p+2]$,
and $ q$ eigenvalues in the interval $[0,4]$
(that include the simple $0$ eigenvalue).
\end{proposition}

\section{Two chains connected to a clique }

We now consider graphs denoted by $G = C_{q_1} \oplus K_p \oplus C_{q_2}$, 
with $p \geq 3$, $q_1$, $q_2 \geq 2$ positive integers, 
defined by the vertex set 
\begin{equation}
\label{Cq1-Kp-Cq2-sites}
{\cal V} =  {\cal V}_{q_1,-} \cup  {\cal V}_{p,-} \cup {\cal V}_{q_2,+}, 
\quad 
{\cal V}_{q_1,-} =  \{- q_1 - p + 2, -q_2 - p + 3,  \ldots,  -p - 1, - p  \}, 
\end{equation}
with ${\cal V}_{p,-}$, ${\cal V}_{q_1,+}$ as in \eqref{Kp-Cq-sites}, 
and a connectivity matrix $c$ that satisfies 
\begin{equation}
\label{Cq1-Kp-Cq2-connectivity-1}
c_{i,j} = -1, \quad \forall i,j \in {\cal V}_{p,-},  
\end{equation}
\begin{equation}
\label{Cq1-Kp-Cq2-connectivity-2}
c_{-p,-p + 1} = c_{-p + 1,-p} = c_{0,1} = c_{1,0} = -1,
\end{equation}
\begin{equation}
\label{Cq1-Kp-Cq2-connectivity-3}
c_{i,j} = -1, \quad \forall i,j \in {\cal V}_{q_1,-} \cup {\cal V}_{q_2,+}  \quad\hbox{with}\quad |i-j| = 1,   
\end{equation}
and $c_{i,j} = 0 $ for all other pairs $(i,j) \in {\cal V} \times {\cal V}$.
We now have 
a complete graph of $p$ nodes, the set ${\cal V}_{p,-} $, joined to two chains 
of $q_1 -1 $, $q_2 -1 $ nodes by   
\eqref{Cq1-Kp-Cq2-connectivity-2}.

\begin{lemma}
\label{weyl-two-leg}
Let $L $ be the graph Laplacian of the graph $ C_{q_1} \oplus K_p \oplus C_{q_2} $, $p \geq 6$,
$q_1$, $q_2 \geq 4$. 
Then $\lambda_1(L)$, $\lambda_2(L)  \in [p,p+2]$, and $\lambda_j(L)= p$,
$\forall j \in \{3, \ldots, p-1\}$.  
Also, $\lambda_p(L)$, $\lambda_{p+1}(L)  
\in (0,\hbox{min} \{ \lambda_1(L_{C_q}) + 2, p \})$ 
($\subset (0,p)$ if $p \geq 6$), $q = \hbox{max}\{q_1,q_2\}$,
and $\lambda_j(L) \in [0,4)$, $\forall j \in \{p+3,\ldots, p+ q_1 + q_2 - 2\}$.
\end{lemma}

\begin{proof}
We will apply the Weyl inequalities \eqref{positive-weyl} using
the decomposition $L = A + B$, 
where $A$ is the $(p + q_1 + q_2 - 2 ) \times (p + q_1 +  q_2 - 2)$
block diagonal matrix with 
blocks $L_{C_{q_1}}$, $L_{K_p}$, and $L_{C_{q_2}}$. 
Then the only non-vanishing elements of 
$B$ are $B(q_1 - 1,q_1 - 1) = B(q_1, q_1)
= B(q_1 + p - 1, q_1 + p - 1) = B(q_1 + p, q_1 + p) = 1$, 
and $B(q_1 ,q_1 - 1) = B(q_1 - 1, q_1) =
B(q_1 + p, q_1 + p-1) = B(q_1 + p - 1,q_1 + p) = -1$.

We then have $\lambda_j(A) = p$ if $j \in \{1, \ldots, p \}$, and 
$\lambda_{p-1+k}(A) = \in (0,4)$ for $k \in \{1, q_1 + q_1 - 4 \}$.
Also  
$ \lambda_{p + q_1 + q_2 - 4}(A)= \lambda_{p + q_1 + q_2 - 3}(A) =
\lambda_{p + q_1 + q_2 -2}(A)  = 0$.
Furthermore, $\lambda_1(B) = \lambda_2(B) = 2$, and $\lambda_j(B) = 0 $,
$\forall j \in \{ 3, \ldots, p + q_1 + q_2 - 2\}$. 

We use $\lambda_j(A) \leq \lambda_j(A + B)$,
$j = 1, \ldots, p + q_1 + q_2 - 2$, 
see \eqref{positive-weyl},
as lower bounds.

For the upper bounds we first have
$$ \lambda_1(A + B) \leq \lambda_1(A ) + \lambda_1(B) = p+2, $$
$$  \lambda_2(A + B) \leq \hbox{min} \{ \lambda_2(A ) + \lambda_1(B), \lambda_1(A ) + \lambda_(B) \} = p+2. $$
For $j \in {3, \ldots, p-1\ }$ we have  
$$ \lambda_j(A + B) \leq \hbox{min} \{\lambda_j(A) + \lambda_1(B),  \lambda_{j-1}(A) +
\lambda_2(B), \ldots, \lambda_1(A) \} = p, $$
using $\lambda_j(B) = 0 $, $\forall j \in \{ 3, \ldots, p + q_1 + q_2 - 2 \}$.  
Also,
\begin{eqnarray}
  \nonumber
 \lambda_p(A + B) &  \leq  &  
\hbox{min}
\{\lambda_p(A) + \lambda_1(B), \lambda_{p-1}(A) + \lambda_2(B), \ldots, \lambda_1(A) \}  \\
\nonumber
& = & 
\hbox{min} \{ \hbox{max}\{\lambda_1(L_{C_{q_1}}), \lambda_1(L_{C_{q_2}}) \} + 2, p\},
\end{eqnarray}
\begin{eqnarray}
  \nonumber
  \lambda_{p+1}(A + B)  &  \leq  &
 \hbox{min}
\{\lambda_{p+1}(A) + \lambda_1(B), \lambda_{p}(A) + \lambda_2(B), \ldots, \lambda_1(A) \}  \\ 
\nonumber
& = &  \hbox{min} \{  
\hbox{max}\{\lambda_1(L_{C_{q_1}}), \lambda_1(L_{C_{q_2}}) \}
 + 2, p \}. 
\end{eqnarray}
If $p \geq 6$ then 
$ \lambda_1(L_{C_q}) + 2 < 6 \leq p$, therefore 
$\lambda_p(A + B)$, $\lambda_{p+1}(A + B) < p$.

For $j \in \{p+2,\ldots, p+ q_1 + q_2 - 5 \}$, i.e. 
$j = p-1+k$, $k \in \{3, \ldots, p + q_1 + q_2 - 4\}$,  
we have
\begin{eqnarray}
  \nonumber
 \lambda_{p-1+k}(A + B) & \leq & 
 \hbox{min} \{ \hbox{max}\{\lambda_{1}(L_{C_{q_1}}), \lambda_{1}(L_{C_{q_1}}) \} + 2 \ldots, \lambda_{1}(A)\}  \\
 \nonumber
 &  = &    
\hbox{min} \{
\hbox{max}\{\lambda_{1}(L_{C_{q_1}}), \lambda_{1}(L_{C_{q_1}}) \} + 2, p \},
\end{eqnarray}

Also 
$$ \lambda_{p+ q_1 + q_2 - 4}(A + B) \leq \min_{l=1,2}\{2, \lambda_{q_j-2}(L_{C_{q_l}}) + 2, 
\lambda_{q_l-3}(L_{C_{q_l}})\}, $$
therefore $ \lambda_{p+ q_1 + q_2 -4}(A + B) < 4 $ if $q_1$, $q_2 \geq 4$
using 
\eqref{chain-eval} for $ \lambda_{q-3}(L_{C_{q}})$, $q \geq 4$, 
and 
$$ \lambda_{p+q_1 + q_2 -3}(A + B) \leq \min_{l=1,2}\{2, \lambda_{q_j-2}(L_{C_{q_l}}) < 4, $$
$ \lambda_{p+q_1 + q_2 - 2}(A + B) = 0$.
The statement follows by combining the above with the lower bounds. 
\end{proof}

\subsection{Clique eigenvalues}

\begin{proposition}
  \label{clique-evect-twoleg-inf}
 Let $L$ be the Laplacian of the graph $ C_{q_1} \oplus K_p \oplus C_{q_2}$, $p \geq 4$, $q_1$, $q_2 \geq 2$.
Then there exists a subspace $E_K$ of dimension $p-3$ such that
$v \in E_K$ satisfies $L v = p v$, and $v_j = 0$, $\forall j \notin \{-p + 2, \ldots, - 1 \}$.
\end{proposition}

\begin{proof}
Let $k \in \{2, \ldots, p-2 \}$ and define the vectors $v^k \in l^2$ by 
\begin{equation}
\label{intern-evec-two-leg}
	v^k_{-1} = 1, \quad v^k_{-k} = -1, \quad v^k_{j} = 0, \quad \forall j \in {\cal V}\setminus\{-1, -k\}. 
\end{equation}
The $v^k$ are linearly independent and 
the statement will follow by checking that 
$L v^k =  p v^k$, $\forall k \in \{2, \ldots, p-2 \}$.
	Let $ i \geq 1$ or $i \leq -(p-1) - 1$. Then $ (L v^k)_{i} = p v^k_{i}$,  
as in \eqref{case-in-out}. The 
case $i \in \{-(p-1), \ldots, 0 \}\setminus\{-1,-k\} $ is   
as in 
\eqref{case-in-out}. The cases $i= -1$, $-k$ follow from 
\eqref{case-on-site-1}, \eqref{case-on-site-2} respectively. 
\end{proof}

As for the chain connected to a clique, the result extends to
the case of two infinite chains $q_1$, $q_2 = \infty$.
\begin{proposition}
  \label{infinite-clique-twoleg} 
Let $L$ be the Laplacian of the graph $ C_{\infty} \oplus K_p \oplus C_{\infty}$, $p \geq 4$. 
Then there exists a subspace $E_K \in l^2$ of dimension $p-3$ such that 
$v \in E_K$ satisfies $L v =  p v$, and $v_j = 0$, $\forall j \notin \{-p + 2, \ldots, - 1 \}$.
\end{proposition}

\subsection{Edge eigenvalues}

Let $v \in l^2$ be an eigenvector of $L$ for
the graph $ C_{q} \oplus K_p \oplus C_{q}$, $q \geq 2$ (including the case $q = \infty$). 
Then $v$ is symmetric if $v_{-p-1 - n} = v_n$, for all integer $n \geq 0 $,
and antisymmetric if  $v_{-p-1 - n} = -v_n$, for all integer $n \geq 0 $.

\begin{proposition}
\label{edge-states-infinite2}
Let $L$ be the Laplacian of the graph $ C_{\infty} \oplus K_p \oplus C_{\infty}$,
    $p \geq 5$, and 
let $E_K$ be as in Proposition \ref{clique-evect-twoleg-inf}. 
Then all 
eigenvectors $v \in l^2 \cap E_K^{\perp}$ of $L$ corresponding to eigenvalues $\lambda > 4$
are either symmetric or antisymmetric. 
$\lambda > 4 $ is the eigenvalue of a symmetric eigenvector $v^S \in l^2 \cap E_K^{\perp}$ 
of $L$ if and 
only if $F_S(\lambda) = 0$, where 
\begin{equation}
\label{two-leg-edge-eigenval-eq-sym}
F_S(\lambda) = (-\lambda + 2) \sigma_+ - (-\lambda + p -1)(-\lambda + 2) + 2(p-2).
\end{equation}
$\lambda > 4 $ is 
the eigenvalue of a symmetric eigenvector $v^A \in l^2 \cap E_K^{\perp}$ of $L$ if and 
only if $F_A(\lambda) = 0$, with
\begin{equation}
\label{two-leg-edge-eigenval-eq-antisym}
F_A(\lambda) =  \sigma_+  - (-\lambda + p +1). 
\end{equation}
$\sigma_+ = \sigma_+(\lambda)$
in \eqref{two-leg-edge-eigenval-eq-sym}, \eqref{two-leg-edge-eigenval-eq-antisym} is as in \eqref{eval-transf-matrix},
Furthermore, both equations $F_A = 0$, $F_S = 0$ have exactly one solution in $(p,p+2)$ and  
no solutions in 
$(4,p] \cup[p+2, + \infty)$. 
\end{proposition}

\begin{proof}
We construct $v \in  l^2$ that satisfy $L v = \lambda v$, with $\lambda > 4 $.
  We first let $v$ be orthogonal to the span of the $p-3$ eigenvectors of
  Proposition \ref{internal-states-infinite} , or 
  \begin{equation*}
- v_{-p+2} + v_{-1} = 0, \quad \ldots - v_{-2} + v_{-1} = 0,  
\end{equation*}
  therefore
  \begin{equation}
   \label{def-c0-two-leg-inf}
v_{-p+2} = \ldots = v_{-1} = C_0   
\end{equation}
for some real $C_0$.
We also let
  \begin{equation}
   \label{def-c1-two-leg-inf}
C_1 = v_{0}, \quad  C_{-1} = v_{-p + 1}.    
\end{equation}
The condition
$L v = \lambda v$ at the nodes $k = -p+2, \ldots, -1$ leads to 
\begin{equation}
(p-3) C_0 + C_1 + C_{-1} - (p-1) C_0  = -\lambda C_0, 
\end{equation}
for all $k = -p+2, \ldots, -1$, or 
\begin{equation}
   \label{first-cond-two-leg-inf}
   C_1 + C_{-1} = (-\lambda + 2)C_0. 
\end{equation}

Let
\begin{equation}
   \label{def-c2-two-leg-inf}
C_2 = v_1, \quad  c_{-2} = v_{-p}.    
\end{equation}
Then 
$L v = \lambda v$ at $k = 0 $ is
  \begin{equation*}
   \label{eigenvaleq-site0-two-leg-inf}
- \sum_{j = - p+1}^{-1} c_{0,j} v_{j} + d_{0} v_{0} - v_{1} = \lambda v_{0},    
\end{equation*}
and reduces to
   \begin{equation}
   \label{second-cond-two-leg-inf}
C_{-1} + (p-2)C_0 + C_2 = (-\lambda + p)C_1.    
\end{equation} 
Similarly, $L v = \lambda v$ at $k = -p+1$ is
  \begin{equation*}
   \label{eigenvaleq-site(-p+1)-two-leg-inf}
- \sum_{j = - p+1}^{-1} c_{-p+1,j} v_{j} + d_{-p+1} v_{-p+1} - v_{-p} = \lambda v_{-p+1},    
\end{equation*}
and reduces to
\begin{equation}
\label{third-cond-two-leg-inf}
C_{1} + (p-2)C_0 + C_{-2} = (-\lambda + p)C_{-1}.    
\end{equation}  

Considering $L v = \lambda v$ at the nodes $j \geq 1$, we argue as in the proof
of Proposition \ref{edge-states-infinite} 
to obtain 
\begin{equation}
\label{cond-four-two-leg-inf}
C_{2} = \sigma_+ C_{1}.    
\end{equation}  

On the other hand, 
$L v = \lambda v$ at the nodes $j \geq -p$
is
\begin{equation}  
\label{def-vgr1-two-leg-inf}
- v_{j-1} + 2 v_{j} - v_{j+1} = \lambda v_{j}, \quad \forall j \leq - p.   
 \end{equation}  
Letting $$ z_j = \begin{pmatrix} v_j \\ v_{j+1} \end{pmatrix}, ~~j \geq -p,$$
\eqref{def-vgr1-two-leg-inf} is equivalent to
\begin{equation}
\label{transf-eq-two-leg-inf}
z_{j+1} = M_{\lambda} z_j, \quad \forall j \leq -p -1,   
\end{equation}
with $M_\lambda$ as in \eqref{def-transf-matrix},
or
\begin{equation}
\label{transf-eq-two-leg-inf}
z_{-p-n} = (M_{\lambda}^{-1})^n z_p, \quad \forall n \geq 0,   
\end{equation}
therefore if $z_{-p} = av^+ + b v^-$, 
$a$, $b$ real, we have 
\begin{equation}
\label{sol-transf-eq-two-leg-inf}
z_{-p-n} = a \sigma_+^{-n} v^+ + b \sigma_-^{-n}v^- =  a \sigma_-^{n} v^+ + b \sigma_+^{n}v^-, \quad \forall, n \geq 0   
\end{equation}
by $\sigma_- \sigma_+ = 1$.
By $\lambda >4$, $|\sigma_-| > 1$, the condition $v \in l^2$ leads to $a = 0$.
We must then require 
$$z_{-p} = \begin{pmatrix} v_{-p} \\ v_{-p+1} \end{pmatrix} = 
b \begin{pmatrix} 1 \\ \sigma_- \end{pmatrix} $$ for some real $b$, or
equivalently 
\begin{equation}
\label{cond-five-two-leg-inf}  
  C_{-2} = \sigma_+ C_{-1} 
\end{equation}
by  \eqref{def-c1-two-leg-inf}, \eqref{def-c2-two-leg-inf}.

The possible eigenvectors $v \in l^2 \cap {E_K}^{\perp}$ of $L$ are determined by equations 
\eqref{first-cond-two-leg-inf},  
\eqref{second-cond-two-leg-inf}, 
\eqref{third-cond-two-leg-inf}, 
\eqref{cond-four-two-leg-inf}, \eqref{cond-five-two-leg-inf} for the $C_{-2}, \ldots, C_{2}$.
We claim that there are only two nontrivial solutions, 
corresponding to $C_{-1} = C_1$ and $C_{-1} = - C_{1}$, leading to 
symmetric and antisymmetric eigenvectors respectively.

To show the claim, we first use \eqref{cond-four-two-leg-inf}, \eqref{cond-five-two-leg-inf} to 
reduce the system to three equations for $C_{-1}$, $C_0$, $C_1$. We 
then add and subtract 
\eqref{second-cond-two-leg-inf}, 
\eqref{third-cond-two-leg-inf},
to obtain 
\begin{equation}
\label{cond-2plus3-two-leg-inf-1}
(C_{-1} + C_1)(-\lambda + p - \sigma_+ -1) = 2(p-2)C_0, 
\end{equation}
and 
\begin{equation}
\label{cond-2minus3-two-leg-inf}
(C_{-1} - C_1)(-1 + \sigma_+ + \lambda - p) =0.
\end{equation}
By \eqref{cond-2minus3-two-leg-inf} we either have $C_{-1} = C_1$, the symmetric 
case, or 
\begin{equation}
\label{antisym-eval-eq-1}
\sigma_+ = -\lambda + p + 1, 
\end{equation}
which by \eqref{cond-2plus3-two-leg-inf-1} implies  
\begin{equation}
\label{cond-2plus3-two-leg-inf-2}
- (C_{-1} + C_1) = (p-2) C_0.
\end{equation}
Suppose that $C_0 \neq 0$, then \eqref{cond-2plus3-two-leg-inf-2},  
\eqref{first-cond-two-leg-inf} imply $\lambda = p$. Then 
\eqref{antisym-eval-eq-1} implies $\sigma_+ = 1$, but this contradicts 
$\sigma_+ \in (-1,0)$,
from \eqref{eval-transf-matrix} with $\lambda > 4 $. 
It follows that $C_0 = 0$. By \eqref{antisym-eval-eq-1} we then have $C_{-1} = - C_1$,
the antisymmetric case.

To see that the corresponding eigenvalues belong to $(p,p+2)$ we first consider 
the antisymmetric case $C_{-1} + C_1 = 0 $. 
The eigenvalue $\lambda$ then   
satisfies \eqref{antisym-eval-eq-1}, with $\sigma_+$ as in \eqref{eval-transf-matrix}.
Let 
\begin{equation}
\label{F-antisym-two-leg-inf}
F_A(\lambda) = \sigma_+ - (-\lambda + p +1).
\end{equation}
$F_A$ is continuous and $ F_A(p) = \sigma_+ - 1 < 0$ 
by $\sigma_+ \in (-1,0)$. Also, 
$ F_A(p+2) = \sigma_+ + 1 >  0$ 
by $\sigma_+ \in (-1,0)$. Therefore we have at least one antisymmetric 
eigenvector with eigenvalue $\lambda \in (p,p+2)$.

We check that  $ F_A'(\lambda) > 0$ for $\lambda \in (p,p+2)$.
Let $x = -\lambda + 2$, and examine ${\tilde F}(x) = F(\lambda(x))$ for 
$x \in (-p, -p+2)$.
We have
$$  {\tilde F}'_A(x) = {\tilde \sigma}'(x) - 1 < 0. $$
We saw in the proof of Proposition \ref{edge-states-infinite}
that $ {\tilde \sigma}'(x) < -1/2$, for all $x \in (-p,-p+2)$, therefore
$F_A'(\lambda) = - {\tilde F}_A'(x)  >  0$, for all $\lambda \in (p,p+2)$.

Also, $\lambda \leq  p$ would imply $\sigma_+  \geq 1$ by \eqref{antisym-eval-eq-1}, contradicting 
$\sigma_+ \in (-1,0)$. Similarly, $\lambda \geq  p + 2$ and \eqref{antisym-eval-eq-1}
would imply $\sigma_+  \leq -1$, a contradiction. 

We conclude 
that equation \eqref{antisym-eval-eq-1} for antisymmetric eigenvectors 
has exactly one solution in $(p,p+2)$ and no other solution
satisfying 
$\lambda > 4$.

In the symmetric case $C_{-1} = C_1$, by 
\eqref{first-cond-two-leg-inf}
we must also have $C_0 \neq 0$, otherwise we have a trivial solution. 
Combining \eqref{cond-2plus3-two-leg-inf-1} and 
\eqref{first-cond-two-leg-inf},
the corresponding eigenvalue $\lambda $ must then satisfy   
\begin{equation}
\label{sym-eigenval-two-leg-inf-1}
(-\lambda + 2)(-\lambda + p - \sigma_+ - 1) = 2(p-2), 
\end{equation}
or equivalently 
\begin{equation}
\label{sym-eigenval-two-leg-inf-2}
\sigma_+ = -\lambda + p -1 - 2\frac{p-2}{-\lambda + 2}. 
\end{equation}
Let 
\begin{equation}
\label{F-sym-two-leg-inf}
F_S(\lambda) = (-\lambda + 2) \sigma_+ - (-\lambda + p -1)(-\lambda + 2) + 2(p-2), 
\end{equation}
with with $\sigma_+$ as in \eqref{eval-transf-matrix}.
$F_S$ is continuous 
and we have 
$$ F_S(p) = (p-2)(2 - \sigma_+) > 0 $$
by $\sigma_+ \in (-1,0)$. Also, 
$$ F_S(p+2) = -p(1 + \sigma_+)- 4 < 0 $$
by $\sigma_+ \in (-1,0)$.
There then at least one symmetric 
eigenvector with eigenvalue $\lambda \in (p,p+2)$.

We see that $ F_S'(\lambda) < 0$ for $\lambda \in (p,p+2)$.
Let $x = -\lambda + 2$, and examine ${\tilde F}(x) = F(\lambda(x))$ for $x \in (-p,-p+2)$.
We have 
$$  {\tilde F}_S'(x) = {\tilde \sigma}_+(x) + x  {\tilde \sigma}'_+(x) - 2x - p + 3. $$
By $x < - p +2$ we have $ - 2 x > 2 p - 4 $,
therefore
$$  {\tilde F}_S'(x) > x  {\tilde \sigma}'_+(x) + p - 1 +  {\tilde \sigma}_+(x). $$
We saw in the proof of Proposition \ref{edge-states-infinite}
that $ {\tilde \sigma}'(x) < -1/2$, for all $x \in (-p,-p+2)$,
therefore
$ x  {\tilde \sigma}'_+(x) > 0$,
$\forall x \in (-p,-p+2)$. Also $p - 1 +  {\tilde \sigma}_+ > p- 2 > 0$.
Thus 
$F_S'(\lambda) =  - {\tilde F}_S'(x)  <  0$, for all $\lambda \in (p,p+2)$.

Also, suppose that $\lambda \leq  p $, then $ -(-\lambda + 2) \leq p-2$ and 
\ref{sym-eigenval-two-leg-inf-2} would imply 
$$ \sigma_+ \geq -1 + 2 \frac{p-2}{\lambda - 2} \geq 1, $$ 
contradicting $\sigma_+ \in (-1,0)$.
Similarly, $\lambda \geq p + 2$ and 
\ref{sym-eigenval-two-leg-inf-2} imply   
$$ \sigma_+ \leq -3 - 2 \frac{p-2}{-\lambda + 2} \leq -3 + 2 \frac{p-2}{p} \leq -1, $$
contradicting $\sigma_+ \in (-1,0)$.

Thus equation \eqref{sym-eigenval-two-leg-inf-1} for symmetric eigenvectors 
has exactly one solution in $(p,p+2)$ and no other solution
satisfying 
$\lambda > 4$.

\end{proof}

\begin{remark}
The eigenvalue corresponding to the antisymmetric eigenvector
is larger than the one for the symmetric eigenvector.
In section 4.4 we derive the asymptotic estimates (\ref{las},\ref{lsym}) 
of these eigenvalues for large $p$. 
\end{remark}

\begin{proposition}
Let $L$ be the Laplacian of 
$ C_{\infty} \oplus K_p \oplus C_{\infty}$, $p \geq 6$.
Then $ \sigma_e({L}) = [0,4]$.
\end{proposition}

The proof uses the argument of Proposition \ref{especinf}.

As for $ K_p \oplus C_\infty $, we have the following theorem
for the spectrum of $ C_{\infty} \oplus K_p \oplus C_{\infty}$, $p \geq 6$.
\begin{proposition}
Let $L$ be the Laplacian of the graph $ C_{\infty} \cup K_p \cup C_{\infty}$, $p \geq 6$.
Then the spectrum of $L $ is a union of the disjoint sets $[0,4]$ (the essential
spectrum of $L$), and $\{\lambda_1, \lambda_2,  p \}$, with 
$ \lambda_1 \geq \lambda_2 \in (p,p+2)$
(the point spectrum of $L$).
\end{proposition}

\begin{remark}
  By Proposition \ref{infinite-clique-twoleg}, 
  for $p=4$ we have three clique eigenvalues 
 $\lambda = 4 \in \sigma_e(L)$.  
  \end{remark}

\subsection{Edge eigenvectors for $ C_{q_1} \oplus K_p \oplus C_{q_2} $}

We have the following proposition showing the existence of two edge
eigenvectors for the graph $ C_{q_1} \oplus K_p \oplus C_{q_2} $.

\begin{proposition}
\label{2-leg-edge-states-finite}
    Let $L$ be the Laplacian of the graph $ C_{q_1} \oplus K_p \oplus C_{q_2}$,
    $p \geq 6$, $q_1$, $q_2 \geq 3$.
Then $\lambda > 4 $
is an eigenvalue of $L$ with corresponding eigenvector $v \in E_K^{\perp}$ 
if and only if 
$D_{q_1,p,q_2}(\lambda) = 0$, where 
\begin{equation}
\label{two-leg-edge-eigenval-eq-finite}
D_{q_1,p,q_2}(\lambda) = (p-2)(2 - Q_{q_1} - Q_{q_2}) - (-\lambda + 2)(Q_{q_1} Q_{q_2} - 1), 
\end{equation}
where
\begin{equation}
\label{constants-Q-G}
Q_{q} = \sigma_+ \frac{1 + \sigma_+^{2 q - 3}}{ 1 + \sigma_+^{2 q - 1}} - (-\lambda + p), 
\end{equation}
and $\sigma_+ = \sigma_+(\lambda)$ is as in \eqref{eval-transf-matrix}.
$D_{q_1,p,q_2}(\lambda)  = 0  $ has exactly two solutions in $(p,p+2]$ and
no solutions in $(4, p]\cup (p+2, +\infty)$.
In the case $q_1 = q_2 = q$, we have the factorization
$ D_{q_1,p,q_2}(\lambda) = - F_{A,q}(\lambda) F_{S,q}(\lambda)$ with 
\begin{equation}
\label{two-leg-edge-finite-sym}
F_{S,q}(\lambda) = (-\lambda + 2)(Q_q + 1) + 2(p-2), \quad F_{A,q}(\lambda)  = 1 - Q_{q}.
\end{equation}
Solutions of $ F_{S,q}(\lambda) = 0$, $ F_{A,q}(\lambda) = 0$ 
correspond to symmetric and antisymmetric eigenvectors of $L$ respectively.
Furthermore, both equations $F_{A,q} (\lambda) = 0$, $F_{S,q} (\lambda)= 0$, $q \geq 3$, have 
exactly one solution in $(p,p+2)$ and no solutions in $(4, p] \cup [p+2,+\infty)$.
\end{proposition}

The proof combines the arguments of Propositions 
\ref{edge-states-infinite2},
\ref{edge-states-finite}
and is 
given in the appendix

We have the following theorem for the whole spectrum of the
graph $ C_{q_1} \oplus K_p \oplus C_{q_2}$.

\begin{proposition}
\label{spectrum-finite2}
Let $L$ be the Laplacian of the graph $ C_{q_1} \oplus K_p \oplus C_{q_2}$, $p \geq 6$, $q_1$, $q_2 \geq 2$. 
Then the spectrum of $L $ consists of the eigenvalue $p$, of multiplicity $p - 3 $, 
two eigenvalues $\lambda_1$, $\lambda_2 \in (p,p+2]$. All other eigenvalues 
are in the interval $[0,4]$ 
\end{proposition}

\subsection{Numerical calculations and asymptotic estimates for large $p$}  

\begin{figure} [H]
\centerline{ \epsfig{file=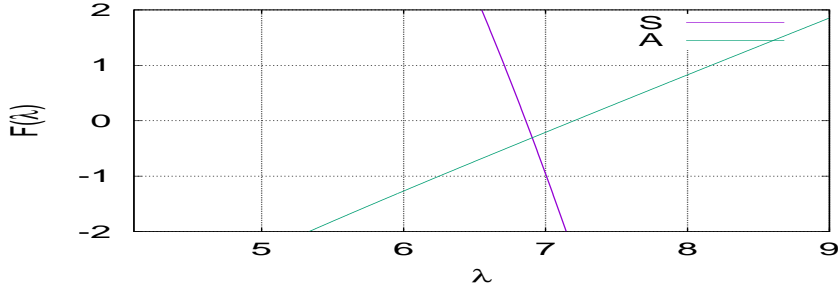,height=4 cm,width=12 cm,angle=0}}
        \caption{Plot of the functions $F_S(\lambda)$ and $F_A(\lambda)$
	for the graph $C_{\infty} \oplus K_6 \oplus C_{\infty}$.}
\label{rbf2}
\end{figure}

We now consider the asymptotic behavior of $\lambda$
in the antisymmetric and symmetric cases.
For that we will use the asymptotic estimates of $\sigma_+$ (\ref{as2})
$$\sigma_+ \approx {1 \over 2 - \lambda}$$

First, assume an antisymmetric solution so that $C_0=0$
Then, $$F_A(\lambda) =  \sigma_+  - (-\lambda + p +1)a=0$$ 
From this equation, we get
$${1 \over 2 - \lambda} = -\lambda + p +1$$
which yields the following second degree equation for $\lambda$
$$\lambda^2 - \lambda (p+3) + 2(p+1) -1 =0 . $$
The interesting solution is
$$\lambda = {p+3 \over 2 } + {1 \over 2} \sqrt{(p+3)^2-4(2p+1)}$$
Expanding the square root, we obtain the final estimate
\be\label{las}\lambda = p+1 + {1 \over p}, \ee
which yields $\lambda=7.16$ for $p=6$ close to the numerical value.
The quantity $\sigma_+$ is
$$\sigma_+ = {p \over p(1-p) -1} .$$
For $p=6$, we have $\sigma_+=-0.19$.

For the symmetric case, we have
$$F_S(\lambda) = (-\lambda + 2) \sigma_+ - (-\lambda + p -1)(-\lambda + 2) + 2(p-2) =0$$
Substituting equation (\ref{as2}) in $F_S(\lambda) =0$ yields
the following second degree equation for $\lambda$
$$\lambda^2 - \lambda (p+1) + 2 =0 . $$
The interesting root is
$$\lambda = {p+1 \over 2 } + {1 \over 2} \sqrt{(p+1)^2-8} .$$
Expanding the square root as above yields the estimate
\be\label{lsym}\lambda = p+1 - {2 \over p+1} . \ee
For $p=6$, we obtain $\lambda =6.71$.

\section{Graphs of complete graphs: results and conjectures}

To conclude the article we present two conjectures on the
spectrum of graphs composed of complete graphs $K_p$ connected
by chains.  

We introduce a graph of complete graphs with the following.
\begin{definition}
A graph of complete graphs is the set ${\cal G}= \{ \tilde V, \tilde E \}$
where $\tilde V = \{K_{p_1}, K_{p_2}, \dots, K_{p_N} \}$ where $K_{p_i}$
is a $p_i$ complete graph. The edges are chains connecting the 
vertices $K_{p_i}$. 
\end{definition}

An example of such a graph of complete graphs is shown in Fig. \ref{ggraph}.
\begin{figure}[H]
\centerline{\epsfig{file=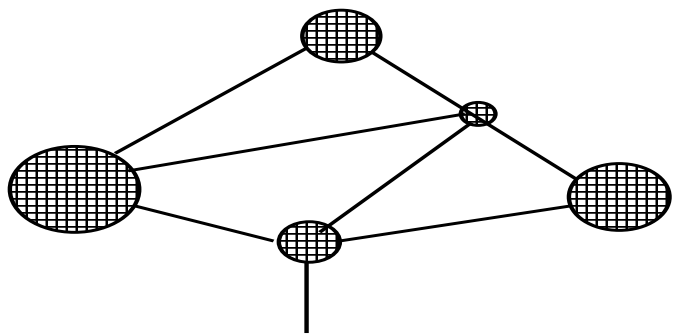,height=5cm,width=10cm,angle=0}}
\caption{A graph of 5 complete graphs connected by chains}
\label{ggraph}
\end{figure}

Assume $p_i \ge 5$ and the length of the edges (chains) greater than 3
and denote $d_i$ the degree of $K_{p_i}$ in ${\cal G}$.
We have the following result.

\begin{proposition} For a graph of complete graphs ${\cal G}$ there are 
at most $p_i -d_i -2 $ clique eigenvalues $p_i$ for $i \in \{1, \dots, N\}$.
There is exactly $p_i -d_i -2 $ clique eigenvalues if $K_{p_i}$ is connected
to edges at different vertices. 
\end{proposition}

The proof is an immediate generalization of the results on the clique
eigenvalues obtained in sections 3 and 4.

We also give the following conjecture.
\begin{proposition} For a graph of complete graphs ${\cal G}$ 
there are $d_i,~ i \in \{1, \dots, N\}$ edge eigenvectors 
with eigenvalues in $]p_i,p_i+2[$ .
\end{proposition}

As a numerical example, we show a graph ${\cal G}$ composed of $K_{10}$
and three chains.

\begin{figure} [H]
\centerline{ \epsfig{file=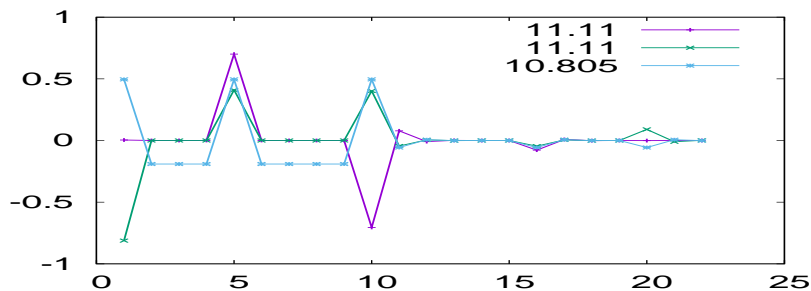,height=4 cm,width=12 cm,angle=0}}
\caption{Plot of the three edge eigenvectors for the graph 
$G=K_{10} \oplus C^1 \oplus C^2 \oplus C^3$. The clique is 
$K_{10}= \{ 1,2,\dots ,10\}$ and the chains are
$C^1 = \{20,21,22 \}$, $C^2 = \{16,17,18,19 \}$ and 
$C^3 = \{ 11,12 ,13,14,15\}$ connected at
vertices $1$, $5$ and $10$ respectively. 
	}
\label{rr5}
\end{figure}
There are $N$ edge eigenvectors of eigenvalue $\lambda > p$. One of them
is symmetric, and $N-1$ are antisymmetric. For the antisymmetric edge 
eigenvector, the eigenvalue is the one calculated for 
the graph with a clique and
two chains. For example for the $G=K_{10} \oplus C^1 \oplus C^2 \oplus C^3$ we have 
$\lambda=11.11$ for the antisymmetric eigenvector.  
Both symmetric and antisymmetric eigenvectors can be labeled
using a three component vector with $\pm 1$, each component
corresponding to a chain connected to $=K_{10}$.
Using this shorthand notation, the symmetric and antisymmetric
eigenvectors are 
$$ s = \begin{pmatrix} 1 \cr 1 \cr 1 \end{pmatrix}, ~~
a^1 = \begin{pmatrix} 0 \cr 1 \cr -1 \end{pmatrix}, ~~
a^2 = \begin{pmatrix} 1 \cr 0 \cr -1 \end{pmatrix},$$
where the antisymmetric eigenvectors correspond to the 
eigenvalue of multiplicity $N-1=2$. These 
eigenvectors are shown in Fig.  \ref{rr5}.

{\bf Acknowledgments.}
We acknowledge the support of a french-mexican Ecos-Conacyt grant. 
JGC thanks IIMAS for financial support and was partially funded 
by ANR grant "Fractal grid". PP thanks Papit IN112119 for partial support.

\section{Appendix: proof of existence of edge eigenvector for $ C_{q_1} \oplus K_p \oplus C_{q_2} $ }

\begin{proof}
To construct $v $ that satisfies
  $L v = \lambda v$, $\lambda > 4$ and is 
orthogonal to the span of the $p-3$ eigenvectors of
Proposition \ref{clique-evect-twoleg-inf} we  
argue as in the proof of 
Proposition \ref{clique-evect-twoleg-inf}.
Considering $L v = \lambda v$ at the sites 
$j = -p + 1, \ldots, 0$, and letting 
$C_{-2} = v_{-p}$, $C_{-1} =  v_{-p+1} $, $C_0 = v_{-p+2} = \ldots  = v_{-1}$,  
$ C_{1} = v_0$, and $ C_2 = v_1$ we obtain 
the equations 
\begin{equation}
   \label{eq1-two-leg-finite}
   C_1 + C_{-1} = (-\lambda + 2)C_0,  
\end{equation}
 \begin{equation}
   \label{eq2-two-leg-finite}
C_{-1} + (p-2)C_0 + C_2 = (-\lambda + p)C_1,     
\end{equation} 
\begin{equation}
\label{eq3-two-leg-finite}
C_{1} + (p-2)C_0 + C_{-2} = (-\lambda + p)C_{-1},
\end{equation}
i.e. as in
\eqref{first-cond-two-leg-inf},
\eqref{second-cond-two-leg-inf},
\eqref{third-cond-two-leg-inf}.

We will express $C_2$, $C_{-2}$ in terms of the $C_1$, $C_{-1}$ respectively 
by analyzing the $L v = \lambda v$ at the remaining sites. 

Consider first 
$L v = \lambda v$ at the nodes $ 1 \leq j \leq q_2 - 2 $, 
\begin{equation}  
\label{two-legs-finite-right}
- v_{j-1} +  2 v_{j} - v_{j+1} = \lambda v_{j},
\quad \forall j \in \{1, \ldots, q_2 -2 \}.
\end{equation}
We argue as in the proof of Proposition \ref{edge-states-finite}. 
Letting $ z_j = (v_{j},v_{j+1})^T$, $ 0 \leq j  \leq q_2 - 1$,  
\eqref{two-legs-finite-right} is equivalent to
\begin{equation}
\label{two-legs-transf-finite-right}
z_{j+1} = M_{\lambda} z_j, \quad \forall j \in \{1, \ldots, q_2-2 \},   
\end{equation}
with $M_\lambda$ as in \eqref{def-transf-matrix}.  
Then $z_0 = a_2 v^+ + b_2 v^- $, $a_2$, $b_2$ real, implies
\begin{equation}
\label{sol-transf-eq-one-leg-finite}
z_{n} = a_2 \sigma_+^n v^+ + b_2 \sigma_-^n v^-, \quad \forall n \in \{ 0, \ldots, q_2 -2\}.   
\end{equation}
Evaluating at $n = q_2 - 2$ and comparing to 
$L v = \lambda v$ at the node $ n = q_2 - 1$, namely
\begin{equation}
v_{q_2 -1} = \frac{ v_{q_2-2}}{-\lambda + 1},  
\end{equation} 
we arrive at 
\begin{equation}
\label{two-leg-last-value}
\frac{1}{-\lambda + 1} [ a_2 \sigma_+^{q_2-2}  + b_2 \sigma_-^{q_2-2}] =  
a_2 \sigma_+^{q_2-1} + b_2 \sigma_-^{q_2-1}.
\end{equation} 
Arguing as in as in the proof of Proposition \ref{edge-states-finite} we have 
\begin{equation}
\label{right-leg-ratio-a2b2}
\frac{b_2}{a_2} =  \sigma_+^{2 q_2},  
\end{equation} 
and $a_2 b_2 \neq 0$.
Comparing expressions  
for $v_0$, $v_1$, and $z_0 = a_2 v^+ + b_2 v^- $, using also \eqref{right-leg-ratio-a2b2} 
for the ratio $b_2/a_2$, we must require  
 \begin{equation}
\label{two-legs-finite-matching-right}
\left[  
\begin{array}{c}
C_1 \\ 
C_2 
\end{array}
\right] = 
a_2
\left( \left[
\begin{array}{c}
1   \\ 
\sigma_+ 
\end{array}
\right] + 
 \sigma_+^{2q_2-1} \left[
\begin{array}{c}
1  \\ 
\sigma_- 
\end{array}
\right]
\right).
\end{equation} 
Then  
\begin{equation}
\label{c1-c2-expressions}
C_1 = a_2 (1 + \sigma_+^{2 q_2 - 1}), \quad C_2 = a_2 (\sigma_+ + \sigma_+^{2 q_2 - 2}),
\end{equation}
and therefore  
\begin{equation}
\label{eliminate-c2}
C_2 = C_1  \sigma_+ \frac{1 + \sigma_+^{2 q_2 - 3}}{ 1 + \sigma_+^{2 q_2 - 1}}.
\end{equation}

Consider now  
$L v = \lambda v$ at the nodes $  - p -  q_1 + 3 \leq j \leq - p  $, 
\begin{equation}  
\label{two-legs-finite-right}
- v_{j-1} + 2 v_{j} - v_{j+1} = \lambda v_{j},
\quad \forall j \in \{-p - q_1 + 3, \ldots, - p \}.
\end{equation}
Letting $ z_j = [v_{j},v_{j+1}]^T$, $ - p - q_1 + 2  \leq j  \leq - p - 1  $,  
\eqref{two-legs-finite-right} is equivalent to
\begin{equation}
\label{two-legs-transf-finite-right}
z_{j+1} = M_{\lambda} z_j, \quad \forall j \in \{ -p - q_1 +2, \ldots, -p -1 \}.   
\end{equation}
Then $ z_{-p - q_1 +2 + n} = M_{\lambda}^n z_{-p - q_1 +2}$, 
$n \in \{0, \ldots, \tilde q_1-2 \}$ and for $n = q_1 - 2$ we have 
\begin{equation}
\label{connect-to-boundary}
z_{-p} = M_{\lambda}^{q_1 - 2} z_{-p - q_1 +2}. 
\end{equation}
Setting $z_{-p} = a_1 v^+ + b_1 v^-$, we then have 
\begin{equation}
\label{iterate-to-boundary}
z_{-p - q_1 +2} = (M^{-1}_{\lambda})^{q_1 - 2}z_{-p} = 
a_1 \sigma_-^{q_1-2} v^+ + b_1 \sigma_+^{q_1 -2}v^-.
\end{equation}
Therefore 
\begin{equation}
\label{two-legs-right-value}
\left[ 
\begin{array}{c}
v_{-p - q_1 + 2} \\ 
v_{-p - q_1 + 3} 
\end{array}
\right] = 
\left[
\begin{array}{c}
a_1 \sigma_-^{q_1-2}  + b_1 \sigma_+^{q_1-2}  \\ 
 a_1 \sigma_-^{q_1 - 3} + b_1 \sigma_-^{q_1  - 3} 
\end{array}
\right]. 
\end{equation} 
At the same time, $L v = \lambda v$ at the node $-p - q_1 +2$ yields 
\begin{equation}
\label{two-leg-first-site}
v_{-p - q_1 + 2} = \frac{ v_{- p - q_1 + 3}}{-\lambda + 1}. 
\end{equation} 
Comparing \eqref{two-legs-right-value}, \eqref{two-leg-first-site}
we must have 
\begin{equation}
\label{first-value-compatibility}
\frac{1}{-\lambda + 1} [ a_1 \sigma_+^{q_1-3}  + b_1 \sigma_-^{q_1-3}] =  
a_1 \sigma_+^{q_1 - 2} + b_1 \sigma_-^{q_1 - 2}.
\end{equation} 
Note that this is \eqref{last-value-one-compatibility} 
in the proof of Proposition \ref{edge-states-finite}
with $q = q_1 - 1$, 
$a = a_1$, $b = b_1$, and $\sigma\pm = \sigma_\mp$. 
Arguing similarly 
we have 
\begin{equation}
\label{left-leg-ratio-a1b1}
\frac{a_1}{b_2} = \sigma_+^{2 q_1-3}. 
\end{equation} 
Comparing expressions  
for $v_{-p}= C_{-2}$, $v_{-p+1} = C_{-1}$, and $z_{-p} = a_1 v^+ + b_2 v^- $, using also \eqref{left-leg-ratio-a1b1} 
for the ratio $a_1/b_2$ we must require  
 \begin{equation}
\label{two-legs-finite-matching-right}
\left[  
\begin{array}{c}
C_{-2} \\ 
C_{-1}
\end{array}
\right] = 
b_1 \left(
  \sigma_+^{2q_1-3}
\left[  
\begin{array}{c}
1   \\ 
\sigma_+ 
\end{array}
\right] + 
\left[
\begin{array}{c}
1  \\ 
\sigma_- 
\end{array}
\right]
\right).
\end{equation} 
Then  
\begin{equation}
\label{cm1-cm2-expressions} 
C_{-2} = b_1 (\sigma_+^{2 q_1 - 3} + 1), \quad C_{-1} =
b_1 (\sigma_+^{2 q_1 - 2} + \sigma_-),
\end{equation}
and 
\begin{equation}
\label{eliminate-c-minus2}
C_{-2} = C_{-1}  \sigma_+
\frac{1 + \sigma_+^{2 q_1 - 3}}{ 1 + \sigma_+^{2 q_1 -1}}.
\end{equation}

By \eqref{eliminate-c2}, 
\eqref{eliminate-c-minus2}, 
system 
\eqref{eq1-two-leg-finite}, 
\eqref{eq2-two-leg-finite}, 
\eqref{eq3-two-leg-finite}
is reduced to  
\begin{eqnarray}
   \label{hom-system-eq1}
   C_1 + C_{-1}  - (-\lambda + 2)C_0 & = & 0, \\
\label{hom-system-eq2}
  Q_{q_2}C_1 + C_{-1} + (p-2) C_0  & = & 0, \\
\label{hom-system-eq3}
  C_{1} +  Q_{q_1}C_{-1} + (p-2)C_0 & = & 0,
\end{eqnarray}
with $Q_{q}$ as in \eqref{constants-Q-G}.
This is a homogeneous system of the form 
$Mx = 0$, with $x = [C_{-1}, C_1, C_0]$, $M$ defined implicitly by
\eqref{hom-system-eq1}-\eqref{hom-system-eq3}.

We compute that $\hbox{det}M = D_{q_1,p,q_2}(\lambda)$, with  
$ D_{q_1,p,q_2}(\lambda)$ given by 
\eqref{two-leg-edge-eigenval-eq-finite}. 
We then have the first statement of the proposition, since 
the trivial solution
of \label{hom-system} would lead to a trivial solution of $L v = \lambda v$.

We now examine $ D_{q_1,p,q_2}(\lambda) = 0$ with $\lambda > 4$.
We first show that there are no solutions in $(4,p]$. 
By \eqref{two-leg-edge-eigenval-eq-finite} $ D_{q_1,p,q_2}(\lambda) = 0$ is 
equivalent to 
\begin{equation}
\label{general-fin-two-leg-eq}
(p-2)(2 - Q_{q_1} - Q_{q_2}) =  (\lambda-2)(1 - Q_{q_1} Q_{q_2}).
\end{equation}
By $\lambda > 4$, we have $\sigma_+ \in (-1,0)$, and 
$G_q = (1 + \sigma_+^{2 {\tilde q} - 1}) (1 + \sigma_+^{2 {\tilde q} + 1})^{-1} \in (0,1)$.
By definition \eqref{constants-Q-G} we then have 
$ Q_{q} \in (-1,0)$;
Assume $  \lambda \in (4,p]$, then the right hand side 
of \eqref{general-fin-two-leg-eq} satisfies
\begin{equation}
\label{rhs-gen-fin-two-leg}
(p-2)(2 - Q_{q_1} - Q_{q_2}) > 2(p-2)(p - \lambda +1) \geq 2(p-2).
\end{equation}
Considering the left hand side 
of \eqref{general-fin-two-leg-eq}, we have  
\begin{equation}
\label{left-gen-fin-two-leg}
1 - Q_{q_1}Q_{q_2} \leq 1
\end{equation}
since 
$$ Q_{q_1}Q_{q_2} = \sigma_+^2 G_{q_1} G_{q_2} +(p-\lambda)
( - \sigma_+ G_{q_1} - \sigma_+ G_{q_2} + p-\lambda) > 0.$$
If $ 1 - Q_{q_1}Q_{q_2} \leq  0$, then the 
left hand side 
of \eqref{general-fin-two-leg-eq} is nonpositive, 
and by \eqref{rhs-gen-fin-two-leg}, equality \eqref{general-fin-two-leg-eq}
can not be satisfied. 
If $ 1 - Q_{q_1}Q_{q_2} > 0$, then the assumption $\lambda \in (4,p]$, 
and \eqref{left-gen-fin-two-leg} imply
$$ (\lambda-2)(1 - Q_{q_1} Q_{q_2}) \leq p-2. $$ 
By $p > 2$ and \eqref{rhs-gen-fin-two-leg} we see that again
\eqref{general-fin-two-leg-eq}
can not be satisfied. 

Combining with 
assumption $p \geq 6$ and Proposition \ref{weyl-two-leg}, 
all solutions of $D_{q_1,p,q_2}(\lambda) = 0$ with $\lambda > 4$
must belong to the interval $(p,p+2]$. 
moreover $ D_{q_1,p,q_2}(\lambda) = 0$ has exactly two solutions in $(p,p+2]$.

We now consider the case $q_1  = q_2 = q$ we have the factorization  
$ D_{q_1,p,q_2}(\lambda) =  - F_{A,q}(\lambda) F_{S,q}(\lambda)$.


We first check that $F_{S,q}(\lambda) = 0$, $F_{A,q}(\lambda) = 0$ correspond 
to symmetric and antisymmetric modes respectively, we add and subtract 
\eqref{hom-system-eq2}, \eqref{hom-system-eq3}
obtaining 
\begin{equation}
\label{sum-hom-system}
(C_1 + C_{-1})[Q_{q} + 1 + 2(p-2)(-\lambda + 2)^{-1}] = 0, 
\end{equation}
\begin{equation}
\label{dif-hom-system}
(C_1 - C_{-1})[Q_{q} - 1] = 0.
\end{equation}

Consider the eigenvalue satisfying  
$ F_{A,q}(\lambda) = Q_{q} - 1 = 0 $.  
Suppose that $C_1 + C_{-1} \neq 0$. Then to satisfy \eqref{sum-hom-system},
we must have 
$ Q_{q} + 1 + 2(p-2)(-\lambda + 2)^{-1} = 0 $, or equivalently to $\lambda = p$. 
Then by the definition of 
$Q_{q}$ in \eqref{constants-Q-G}, $ Q_{q} - 1 = 0 $ is  
$$ \sigma_+ =
\frac{1 + \sigma_+^{2 q - 1}}{1 + \sigma_+^{2 q - 3}}. $$
By $\lambda > 4$ we have $\sigma_+ \in (-1,0)$. On the other hand 
$$ \frac{1 + \sigma_+^{2 {\tilde q} + 1}}{1 + \sigma_+^{2 {\tilde q} - 1}} = 
\frac{1 - |\sigma_+|^{2 {\tilde q} + 1}}{1 - |\sigma_+|^{2 {\tilde q} - 1}} > 1,$$ 
thus $ Q_{q} - 1 \neq 0 $. We therefore have $C_1 + C_{-1} = 0$, and 
$C_0 = 0$ by \eqref{hom-system-eq1}. 
By \eqref{c1-c2-expressions}, \eqref{cm1-cm2-expressions} we also have $C_{-2} = -C_2$, 
and by $L v = \lambda v$ at sites $ j \geq 1$, $j \leq - p$ of the graph
we obtain $v_{1 + n} = - v_{-p -n}$, $\forall n \in \{1, \ldots, q-2\}$.
Thus the corresponding eigenvector is antisymmetric. 

Consider the eigenvalue satisfying 
$ F_{A,q}(\lambda) = Q_{q} + 1 + 2(p-2)(-\lambda + 2)^{-1}= 0 $. By the previous argument 
this is can not hold if $Q_{q} - 1 = 0 $. Thus $C_1 - C_{-1} = 0$. 
By \eqref{c1-c2-expressions}, \eqref{cm1-cm2-expressions}
we also have $C_{-2} =  C_2$,  
and by $L v = \lambda v$ at sites $ j \geq 1$, $j \leq - p$ of the graph
we obtain $v_{1 + n} = v_{-p -n}$, $\forall n \in \{1, \ldots, {\tilde q}-1\}$.
Thus the corresponding eigenvector is symmetric. 


To see that we have exactly one symmetric and one antisymmetric
eigenvector, 
we observe that by \eqref{two-leg-edge-finite-sym},
$\sigma_+ \in (-1,0)$,
$$ F_{A,q}(p) =
\sigma_+ \frac{1 + \sigma_+^{2q-3}}{ 1 + \sigma_+^{2q-1}} - 1
< 0, $$ 
and
$$ F_{A,q}(p+2) =
\sigma_+ \frac{1 + \sigma_+^{2q-3}}{ 1 + \sigma_+^{2q-1}} + 1 > 0, $$
assuming $ q \geq 2$.
$F_{A,q}(\lambda)$ therefore has at least one root in $(p,p+2)$.
Also, 
$$ F_{S,q}(p) = (p -2) \left[ \sigma_+
  \frac{1 -  \sigma_+^{2q-3}}{ 1 + \sigma_+^{2q-1}} + 1 \right] > 0, $$
and
$$ F_{q,S}(p+2) ) = - p \left[ \sigma_+
  \frac{1 -  \sigma_+^{2q-3}}{ 1 + \sigma_+^{2q-1}} + 1 \right] < 0, $$
assuming $q \geq 2$.
$F_{S,q}(\lambda)$ therefore has at least one root in $(p,p+2)$.
By the count of the roots of $D_{q_1,p,q_2}(\lambda)$ for $\lambda > 4$ above,  
there exist unique $\lambda_A$, $\lambda_S \in (p,p+2)$
satisfying $F_{A,q}(\lambda_A) = 0$, $F_{S,q}(\lambda_S) = 0 $ respectively,
moreover
these are the only roots of $F_{A,q}(\lambda_A)$, $F_{S,q}(\lambda_S)$
with $\lambda > 4$.

\end{proof}

\end{document}